\documentclass{amsart}
\usepackage{amssymb,latexsym,graphics}
\usepackage[mathscr]{eucal}
\usepackage{enumerate}

\newtheorem{thm}{Theorem}[section]
\newtheorem{cor}[thm]{Corollary}
\newtheorem{lem}[thm]{Lemma}
\newtheorem{prop}[thm]{Proposition}

\theoremstyle{remark}
\newtheorem{rem}[thm]{Remark}

\theoremstyle{definition}

\newcommand{\step}[2]{\noindent\textbf{Step #1:}(\emph{#2})}
\newcommand{\case}[2]{\noindent\textbf{Case #1:}(\emph{#2})}

\newcommand{\F}{\mathcal F}
\newcommand{\I}{\mathcal I}
\newcommand{\E}{\mathcal E}
\newcommand{\tC}{{\widetilde C}}

\newcommand{\Z}{{\mathbb Z}}
\renewcommand{\S}{{\mathbb S}}
\newcommand{\N}{{\mathbb N}}
\newcommand{\M}{{\mathcal{M}}}
\renewcommand{\H}{{\mathbb H}}
\newcommand{\R}{{\mathbb R}}

\newcommand{\lp}[2]{\Vert \, #1 \, \Vert_{#2}}

\newcommand{\td}{\widetilde}

\newcommand{\dint}{ {\int\!\!\int} }
\newcommand{\snabla}{ {\slash\!\!\! \nabla} }
\newcommand{\spartial}{ {\slash\!\!\! \partial} }
\newcommand{\bL}{ {\underline{L}} }

\newcommand{\ret}{\vspace{.2cm}}

\begin{document}

\title[Large data Wave Maps]
{ Regularity of Wave-Maps in dimension $2+1$}
\author{Jacob Sterbenz}
\address{Department of Mathematics,
University of California, San Diego, CA 92093-0112}
\email{jsterben@math.ucsd.edu}
\thanks{The first author was supported in part by the NSF grant DMS-0701087}
\author{Daniel Tataru}
\address{Department of Mathematics, University of California, Berkeley, CA 94720-3840}
\email{tataru@math.berkeley.edu}
\thanks{The second author was
  supported in part by the NSF grant DMS-0801261.}
\date{}
\dedicatory{}
\commby{}


\begin{abstract}
  In this article we prove a Sacks-Uhlenbeck/Struwe type
  global regularity result for wave-maps
  $\Phi:\mathbb{R}^{2+1}\to\mathcal{M}$ into general compact target
  manifolds $\mathcal{M}$.
\end{abstract}

\maketitle

\tableofcontents


\section{Introduction}

In this article we consider  large data Wave-Maps from
$\R^{2+1}$ into a compact Riemannian manifold $(\M,m)$, and we prove
that regularity and dispersive bounds persist as long as a
soliton-like concentration is absent. This is a companion to our
concurrent article \cite{STed}, where the same result is proved under
a stronger energy dispersion assumption, see Theorem~\ref{ed} below.

The set-up we consider is the same as the one in \cite{Tataru_WM1}, using
the extrinsic formulation of the Wave-Maps equation.
Precisely, we consider the target manifold $(\mathcal{M},m)$ as an isometrically
embedded submanifold of $\R^N$. Then we can view the $\mathcal{M}$ valued
functions as $\R^N$ valued functions whose range is contained in $\mathcal{M}$.
Such an embedding always exists by Nash's theorem~\cite{MR17:782b}
(see also Gromov~\cite{MR43:1212} and G\"unther~\cite{MR93b:53049}).
In this context the Wave-Maps equation can be expressed in a form
which involves the second fundamental form $\mathcal{S}$ of $\mathcal{M}$, viewed as a
symmetric bilinear form:
\begin{align}
    \mathcal{S}:\, T\mathcal{M} \times T\mathcal{M}  \ &\to  \ N\mathcal{M}\ ,
    &\langle S(X,Y),N\rangle  \ &=  \ \langle \partial_X
    N, Y\rangle \ . \notag
\end{align}
The Cauchy problem for the wave maps equation
has the form:
\begin{subequations}\label{main_eq}
\begin{align}
    \Box \phi^a  \ &=  \ - \mathcal{S}^{a}_{bc}(\phi) \partial^\alpha \phi^b
        \partial_\alpha\phi^c \ , &(\phi^1,\ldots,\phi^N) \ := \ \Phi  \ , \label{main_eq1}\\
    \Phi(0,x)\ &= \ \Phi_0(x) \ , \ \ \partial_t\Phi(0,x) \ = \ \Phi_1(x) \ . \label{main_eq2}
\end{align}
\end{subequations}
where the  initial data $(\Phi_0,\Phi_1)$ is  chosen to obey the constraint:
\begin{align}
    \Phi_0(x) \ &\in\  \mathcal{M} \ ,
    &\Phi_1(x) \ &\in \ T_{\Phi_0(x)}\mathcal{M} \ ,
    &x &\in \R^2 \ . \notag
\end{align}
In the sequel, it will be convenient for us to use the notation
$\Phi[t] = (\Phi(t), \partial_t \Phi(t))$.

There is a conserved energy for this problem,
\[
\E[\Phi](t) = \frac12 \int (|\partial_t \Phi|^2 + |\nabla_x \Phi|^2) dx \ .
\]
This is  invariant with respect to the scaling that preserves the
equation, $\Phi(x,t) \to \Phi(\lambda x,\lambda t)$. Because of this, we say that the
problem is {\bf energy critical}.

The present article contributes to the understanding of finite
energy solutions with arbitrarily large initial data, a problem
which has been the subject of intense investigations for some time
now. We will not attempt to give a detailed account of the history
of this subject here. Instead,  we refer the reader to the surveys
\cite{MR2043751} and \cite{MR2488946}, and the references therein.
The general (i.e. without symmetry assumptions) small energy problem
for compact targets was initiated in the ground-breaking work of
Klainerman-Machedon \cite{MR94h:35137}--\cite{MR1381973}, and
completed by the work of Tao~\cite{Tao_WM} when $\M$ is a sphere,
and Tataru~\cite{Tataru_WM1} for general isometrically embedded
manifolds; see also the work of Krieger~\cite{MR2094472} on the
non-compact hyperbolic plane, which was treated from the intrinsic
point of view.

At a minimum one expects
the solutions to belong to the space $C(\R; \dot H^1(\R^2)) \cap \dot
C^1(\R,L^2(\R^2))$. However, this information does not suffice in
order to study the equation and to obtain uniqueness statements. Instead,
a smaller Banach space $S \subset C(\R; \dot H^1(\R^2)) \cap \dot
C^1(\R,L^2(\R^2))$ was introduced in \cite{Tao_WM}, modifying an earlier
structure in \cite{Tataru_WM2}. Beside the energy, $S$ also contains
Strichartz type information in various frequency localized
contexts. The full description of $S$ is not necessary here. However,
we do use the fact that $S \cap L^\infty$ is an algebra, as well as
the $X^{s,b}$ type embedding $ S \subset \dot X^{1,\frac12}_{\infty}$.
See our companion paper \cite{STed} for precise definitions. 
The standard small data result is as follows:

\begin{thm}[\cite{Tao_WM},\cite{Tataru_WM1},\cite{MR2094472}]
There is some $E_0 > 0$ so that for each smooth initial data
$(\Phi_0,\Phi_1)$ satisfying $\E[\Phi](0) \leq E_0$ there exists
a unique global smooth solution $\Phi = T(\Phi_0,\Phi_1) \in S$.
In addition, the above solution operator $T$ extends to a continuous operator from $\dot H^1 \times L^2$ to $S$ with
\begin{equation}
\| \Phi\|_{S} \lesssim \|(\Phi_0,\Phi_1)\|_{\dot{H}^1 \times L^2} \ . \notag
\end{equation}
 Furthermore, the following weak
Lipschitz stability estimate holds for these solutions for $s <1$ and close to $1$:
\begin{equation}
    \| \Phi - \Psi\|_{C(\R; \dot H^s) \cap \dot C^1(\R,\dot H^{s-1})}
    \lesssim \| \Phi[0] - \Psi[0]\|_{\dot H^s \times \dot H^{s-1}}
     \ . \label{weaklip}
\end{equation}
\label{tsmall}\end{thm}

Due to the finite speed of propagation, the corresponding local result
is also valid, say with the initial data in a ball and the solution in
the corresponding domain of uniqueness.

The aim of the companion \cite{STed} to the present work is to provide
a conditional $S$ bound for large data solutions, under a weak energy
dispersion condition. For that, we have introduced the notion of the
{\em energy dispersion} of a wave map $\Phi$ defined on an interval $I$,
\[
 ED[\Phi] = \sup_{k \in \Z} \| P_k \Phi\|_{L^\infty_{t,x}[I]}
\]
where $P_k$ are spatial Littlewood-Paley projectors at frequency $2^k$.
The main result in \cite{STed} is as follows:

\begin{thm}[Main Theorem in \cite{STed}]
For each $E > 0$ there exist $F(E)>0$ and $\epsilon(E) > 0$
so that  any wave-map $\Phi$ in a time interval $I$ with energy
$\E[\Phi] \leq E$ and energy dispersion $ED[\Phi] \leq \epsilon(E)$
must satisfy $\|\Phi\|_{S[I]} \leq F(E)$.
\label{ed}\end{thm}

In particular, this result implies that a wave map with energy $E$ cannot
blow-up as long as its energy dispersion stays below $\epsilon(E)$.

In this article we establish unconditional analogs of the above result,
in particular settling the blow-up versus global
regularity and scattering question for the large data problem.
We begin with some notations. We consider the forward light cone
\[
C = \{ 0 \leq t < \infty, \ r \leq t \}
\]
and its subsets
\[
 C_{[t_0,t_1]} = \{ t_0 \leq t \leq t_1, \ r \leq t \} \ .
\]
The lateral boundary of $C_{[t_0,t_1]}$ is denoted by
 $\partial C_{[t_0,t_1]}$.
The time sections of the cone are denoted by
\[
S_{t_0} = \{t = t_0, \ |x| \leq t\} \ .
\]
We also use the translated cones
\[
C^\delta = \{ \delta \leq t < \infty, \ r \leq t-\delta \}
\]
as well as the corresponding notations
$ C_{[t_0,t_1]}^\delta$,  $\partial C_{[t_0,t_1]}^\delta$ and $S_{t_0}^\delta$
for $t_0  > \delta$.

Given a wave map $\Phi$ in $C$ or in a subset $ C_{[t_0,t_1]}$
of it we define  the energy of $\Phi$ on time sections as
\[
\E_{S_t}[\Phi] = \frac12 \int_{S_t}
(|\partial_t  \Phi|^2 + |\nabla_x \Phi|^2)
dx \ .
\]
 It is convenient
to do the computations in terms of the null frame
\begin{align}
  L \ &= \ \partial_t + \partial_r \ , &\bL \ &= \ \partial_t -
  \partial_r \ , &\spartial \ &= \ r^{-1}\partial_\theta \ . \notag
\end{align}
We define the flux of $\Phi$ between $t_0$ and $t_1$
as
\[
\F_{[t_0,t_1]}[\Phi] =  \int_{\partial C_{[t_0,t_1]}}
 \big( \frac{1}{4}|L \Phi|^2  +  \frac{1}{2}|\spartial \Phi|^2\big) dA \ .
\]

By standard energy estimates we have the energy conservation relation
\begin{equation}
    \E_{S_{t_1}}[\Phi] = \E_{S_{t_0}}[\Phi] +  \F_{[t_0,t_1]}[\Phi]
    \ . \label{en-flux}
\end{equation}
This shows that
$\E_{S_t}[\Phi]$ is a nondecreasing function of $t$.

\subsection{The Question of Blowup}

We begin with the blow-up question.
A standard argument which uses the small data result and the finite
speed of propagation shows that if blow-up occurs then it must occur
at the tip of a light-cone where the energy (inside the cone) concentrates.
 After a translation and rescaling it suffices to consider wave maps
 $\Phi$ in the cone $C_{[0,1]}$. If
\[
\lim_{t \to 0} \E_{S_t}[\Phi] \leq E_0
\]
then blow-up cannot occur at the origin due to the local result,
and in fact it follows that
\[
\lim_{t \to 0} \E_{S_t}[\Phi] = 0 \ .
\]
Thus the interesting case is when we are in an energy concentration scenario
\begin{equation}
\lim_{t \to 0} \E_{S_t}[\Phi] > E_0 \ .
\label{econ}\end{equation}
The main result we  prove here is the following:

\begin{thm}
Let $\Phi: C_{(0,1]} \to \M$ be a $C^\infty$ wave map. Then exactly
one of the following possibilities must hold:

\begin{enumerate}[A)]
\item There exists a sequence of points $(t_n,x_n) \in C_{[0,1]}$ and
 scales $r_n$ with
\[
(t_n,x_n) \to (0,0) \ , \qquad \limsup \frac{|x_n|}{t_n} < 1 \ , \qquad
 \lim \frac{r_n}{t_n} = 0
\]
so that the  rescaled sequence of wave-maps
\begin{equation}
    \Phi^{(n)}(t,x) \ = \  \Phi\big(t_n+r_n t,x_n + r_n x\big) \ ,
\label{rescaling}
\end{equation}
converges strongly in $H^1_{loc}$ to a Lorentz transform of an entire
Harmonic-Map of
nontrivial energy:
\begin{equation}
    \Phi^{(\infty)}:\mathbb{R}^{2}\to \M \ , \qquad
    0 <   \lp{\Phi^{(\infty)}}{\dot{H}^1(\mathbb{R}^2)}
    \leqslant \lim_{t \to 0} \E_{S_t}[\Phi]
    \ . \notag
\end{equation}

\item For each $\epsilon > 0$ there exists $0< t_0 \leqslant 1$ and a
wave map extension
\[
\Phi: \R^2 \times (0,t_0] \to \M
\]
with bounded energy
\begin{equation}
\E[\Phi] \leq (1+ \epsilon^8) \lim_{t \to 0} \E_{S_t}[\Phi]
\label{caseben}\end{equation}
and energy dispersion,
\begin{equation}
\sup_{t \in (0,t_0]} \sup_{k \in \Z}\big( \| P_k \Phi(t)\|_{L^\infty_x} + 2^{-k}
\|P_k \partial_t \Phi(t)\|_{L^\infty_x}\big) \leq \epsilon \ .
\label{casebed}\end{equation}
\end{enumerate}
\label{maint}\end{thm}

We remark that a nontrivial harmonic map
$\Phi^{(\infty)}:\mathbb{R}^{2}\to \M$ cannot have an arbitrarily
small energy. Precisely, there are two possibilities.  Either there
are no such harmonic maps (for instance, in the case when $\M$ is
negatively curved, see \cite{Lem}) or there exists a lowest energy
nontrivial harmonic map, which we denote by $\E_{0}(\M) >
0$. Furthermore, a simple computation shows that the energy of any
harmonic map will increase if we apply a Lorentz
transformation. Hence, combining the results of Theorem~\ref{maint}
and Theorem~\ref{ed} we obtain the following:

\begin{cor}(Finite Time Regularity for Wave-Maps)
The following statements hold:

\begin{enumerate}[A)]

\item Assume that $\M$ is a compact Riemannian manifold so that there
  are no nontrivial finite energy harmonic maps
  $\Phi^{(\infty)}:\mathbb{R}^{2}\to \M$. Then for any finite energy
  data $\Phi[0]:\R^2\times \R^2\to \M\times T\M$ for the wave map
  equation \eqref{main_eq} there exists a global solution $\Phi \in
  S(0,T)$ for all $T > 0$. In addition, this global solution retains
  any additional regularity of the initial data.

\item Let $\pi:\td{\M}\to \M$ be a Riemannian covering, with $\M$ compact, and
such that there are no nontrivial finite energy harmonic  maps
$\Phi^{(\infty)}:\mathbb{R}^{2}\to \M$.
If $\Phi[0]:\R^2\times \R^2\to
\td{\M}\times T\td{\M}$ is $C^\infty$, then there is a  global  $C^\infty$
solution to $\td{\M}$ with this data.

\item Suppose that there exists a lowest energy nontrivial harmonic map
into $\M$
with energy $\E_{0}(\M)$. Then for any data $\Phi[0]:\R^2\times \R^2\to
\M\times T\M$ for the wave map
equation \eqref{main_eq} with energy below $\E_{0}(\M)$, there exists a
global solution $\Phi \in S(0,T)$ for all $T > 0$.
\end{enumerate}
\label{tccor}\end{cor}

We remark that the  statement in part B) is a simple consequence of A)
and restricting the  projection $\pi\circ\Phi$ to a sufficiently
small section $S_t$ of a cone where one expects blowup of the
original map into $\td{\M}$. In particular, since this projection is
regular by part A), its image lies in a simply
connected set for sufficiently small $t$. Thus, this projection can
be inverted to yield regularity of the original map close to the
suspected blowup point. Because of this trivial reduction, we work
exclusively with compact $\M$ in the sequel. It should be remarked
however, that as a (very) special case of this result one obtains
global regularity for smooth Wave-Maps into all hyperbolic spaces
$\mathbb{H}^n$, which has been a long-standing and important
conjecture in geometric wave equations due to its relation with
problems in general relativity (see Chapter 16 of \cite{MR2473363}).

The statement of Corollary \ref{tccor} in its full generality was
known as the \emph{Threshold Conjecture}. Similar results were
previously established for the Wave-Map problem via symmetry
reductions in the works \cite{MR94e:58030}, \cite{MR96c:58049},
\cite{MR1971037}, and \cite{MR1990477}. General results of this
type, as well as fairly strong refinements, have been known for the
Harmonic-Map heat-flow for some time (see \cite{MR826871} and
\cite{MR1438148}). As with the heat flow, Theorem~\ref{maint} does
not prevent the formation of multiple singularities on top of each
other. To the contrary, such bubble-trees are to be expected (see
\cite{MR1735698}).

Finally, we remark that this result is sharp. In the case of $M = \S^2$
there exists a lowest energy nontrivial harmonic map, namely the
stereographic projection $Q$. The results in \cite{KST} assert that
blow-up with a rescaled $Q$ profile can occur for initial data with
energy arbitrarily close to $\E[Q]$. We also refer the reader to
\cite{RS} for blow-up results near higher energy harmonic maps.

\subsection{The Question of Scattering}

Next we consider the scattering problem, for which we start with a
finite energy wave map $\Phi$ in $\R^2 \times [0,\infty)$ and consider
its behavior as $t \to \infty$. Here by scattering we simply mean the
fact that $\Phi \in S$; if that is the case, then the structure theorem for
large energy Wave-Maps in
\cite{STed}  shows that $\Phi$ behaves at $\infty$ as a
linear wave after an appropriate renormalization.

We can select a ball $B$ so that outside $B$ the energy is small,
$\E_{B^c}[ \Phi] < \frac1{10} E_0$.  Then outside the influence cone of
$B$, the solution $\Phi$ behaves like a small data wave map. Hence it
remains to study it within the influence cone of $B$. After scaling
and translation, it suffices to work with wave maps $\Phi$ in the
outgoing cone $C_{[1,\infty)}$ which have finite energy, i.e.
\begin{equation}
\lim_{t \to \infty} \E_{S_t}[\Phi] < \infty \ .
\label{eout}\end{equation}
We prove the following result:

\begin{thm}
Let $\Phi: C_{[1,\infty)} \to \M$ be a $C^\infty$ wave map which
satisfies \eqref{eout}. Then exactly
one of the following possibilities must hold:

\begin{enumerate}[A)]
\item There exists a sequence of points $(t_n,x_n) \in C_{[1,\infty)}$ and
 scales $r_n$ with
\[
t_n \to \infty \ , \qquad \limsup \frac{|x_n|}{t_n} < 1 \ , \qquad
 \lim \frac{r_n}{t_n}  = 0
\]
so that the  rescaled sequence of wave-maps
\begin{equation}
    \Phi^{(n)}(t,x) \ = \  \Phi\big(t_n+r_n t,x_n + r_n x\big) \ ,
\label{rescalings}
\end{equation}
converges strongly in $H^1_{loc}$ to a Lorentz transform of an entire
Harmonic-Map of
nontrivial energy:
\begin{equation}
 \Phi^{(\infty)}:\mathbb{R}^{2}\to \M \ , \qquad   0 <
 \lp{\Phi^{(\infty)}}{\dot{H}^1(\mathbb{R}^2)} \leqslant \lim_{t \to \infty} \E_{S_t}[\Phi] \ .
 \notag
\end{equation}

\item For each $\epsilon > 0$ there exists $t_0 > 1$ and a
wave map extension
\[
\Phi: \R^2 \times [t_0,\infty) \to \M
\]
with bounded energy
\begin{equation}
   \E[\Phi]\leq(1+\epsilon^8)\lim_{t \to \infty} \E_{S_t}[\Phi]
\label{casebens}\end{equation} and energy dispersion,
\begin{equation}
    \sup_{t \in [t_0,\infty)} \sup_{k \in \Z}\big( \| P_k \Phi(t)\|_{L^\infty_x} +
    2^{-k} \|P_k \partial_t \Phi(t)\|_{L^\infty_x}\big) \leq \epsilon \ .
    \label{casebeds}
\end{equation}
\end{enumerate}
\label{maints}\end{thm}

In case B) Theorem~\ref{ed} then implies that scattering holds as $t
\to \infty$.  Thus if scattering does not hold then we must be in case
A). As a corollary, it follows that scattering can only fail for
wave-maps $\Phi$ whose energy satisfies
\begin{equation}
\E[\Phi] \geq \E_0(\M) \ .
\end{equation}
Thus Corollary \ref{tccor} can be strengthened to

\begin{cor}[Scattering for Large Data Wave-Maps]
The following statements hold:

\begin{enumerate}[A)]
\item Assume that there are no nontrivial finite energy harmonic  maps
$\Phi^{(\infty)}:\mathbb{R}^{2}\to \M$. Then for any finite energy data  $\Phi[0]$
for the wave map equation \eqref{main_eq} there exists a global
solution $\Phi \in S$.

\item Suppose that there exists a lowest energy nontrivial harmonic map,
with energy $\E_{0}(\M)$. Then for any data $\Phi[0]$ for the wave
map equation \eqref{main_eq} with energy below $\E_{0}(\M)$ there
exists a global solution $\Phi \in S$.
\end{enumerate}
\label{tcthm}\end{cor}

Ideally one would also like to have a constructive bound of the form
\begin{equation}
 \| \Phi\|_{S} \leq F( \E[\Phi]) \ . \notag
\end{equation}
This does not seem to follow directly from our results. Furthermore,
our results do not seem to directly imply scattering for non-compact
targets in the absence of harmonic maps (only scattering of the
projection). Results similar to Corollary \ref{tcthm} were
previously established in spherically symmetric and equivariant
cases, see \cite{MR94j:58044} and \cite{CKM}.\ret

Finally, we would like to remark that results similar in spirit to
the ones of this paper and \cite{STed} have been recently announced.
In the case where $\mathcal{M}=\mathbb{H}^n$, the hyperbolic spaces,
globally regularity and scattering follows from the program of Tao
\cite{Tao_LWM1}, \cite{Tao_LWM2}, \cite{Tao_LWM3}, \cite{Tao_LWM4}, \cite{Tao_LWM5}
and \cite{Tao_LWM6}. In the case where the target
$\mathcal{M}$ is a negatively curved Riemann surface, Krieger and Schlag \cite{KSc}
provide global regularity and scattering via a modification of the
Kenig-Merle method \cite{KM},
which uses as a key component suitably defined Bahouri-Gerard \cite{BG}
type decompositions. \ret

\textbf{Acknowledgements:} The authors would like to thank Manos Grillakis, Sergiu Klainerman,
 Joachim Krieger, Matei Machedon,  Igor Rodnianski, and Wilhelm Schlag
for many stimulating discussions over the years regarding the wave-map
problem. We would also especially like to thank Terry Tao
for several key discussions on the nature of induction-on-energy
type proofs.\ret

\section{Overview of the Proof}
The proofs of Theorem~\ref{maint} and Theorem~\ref{maints}
are almost identical.  The three main building blocks of both proofs
are (i) weighted energy estimates, (ii) elimination of finite energy self-similar solutions,
and (iii) a compactness result.

Our main energy estimates are established in Section~\ref{energy_sect}.
Beside the standard energy bounds involving the
$\partial_t$ vector field we also use the vector field
\begin{equation}
  X_0 \ = \ \frac{1}{\rho}\left(t\partial_t
    + r\partial_r \right), \qquad \rho = \sqrt{t^2-r^2}
 \label{X_field0}
\end{equation}
as well as its time translates. This leads to a family of weighted
energy estimates, see  \eqref{enxe} below, which has appeared in
various guises in the literature. The first such reference we are
aware of is the work of Grillakis \cite{MR1777636}. Our approach is
closest to the work of Tao \cite{Tao_LWM1} and \cite{Tao_LWM2} (see
also Chapter 6.3 of \cite{MR2233925}). These bounds are also
essentially identical to the ``rigidity estimate'' of Kenig-Merle
\cite{KM}. It should be noted that estimates of this type are
probably the \emph{only} generally useful time-like component
concentration bounds possible for non-symmetric wave equations, and
they will hold for any Lagrangian field equation on $(2+1)$
Minkowski space.

Next, we introduce a general argument to rule out the existence of
finite energy self-similar solutions to \eqref{main_eq}. Such
results are essentially standard in the literature (e.g. see the
section on wave-maps in \cite{MR1674843}), but we take some care
here to develop a version which applies to the setup of our work.
This crucially uses the energy estimates developed in
Section~\ref{energy_sect}, as well as a boundary regularity result
of J. Qing  for harmonic maps (see \cite{MR1223710}).

The compactness result in Proposition~\ref{pcompact}, proved in
Section~\ref{X_sect}, allows us to produce the strongly convergent
subsequence of wave maps in case A) of Theorems~\ref{maint},
\ref{maints}. It applies to local sequences $\Phi^{(n)}$ of small
energy wave maps with the additional property that $X \Phi^{(n)} \to
0$ in $L^2$ for some time-like vector field $X$. This estimate uses
only the standard small energy theory of \cite{Tataru_WM1}, and is
completely independent of the more involved regularity result in our
companion paper \cite{STed}.

Given these three building blocks, the proof of
Theorems~\ref{maint} and \ref{maints} presented in
Section~\ref{theproof} proceeds as follows:\ret


\step{1}{Extension and scaling}  We assume that part B) of
Theorem~\ref{maint}, respectively Theorem~\ref{maints} does not hold
for a wave map $\Phi$ and for some $\epsilon > 0$. We construct an
extension of $\Phi$ as in part B) satisfying \eqref{caseben},
respectively \eqref{casebens}.  Then the energy dispersion relation
\eqref{casebed}, respectively \eqref{casebeds} must fail. Thus, we can
find sequences $t_n$, $x_n$, and $k_n$ so that
\begin{equation}
    | P_{k_n} \Phi(t_n,x_n)| + 2^{-k_n} |P_{k_n} \partial_t \Phi(t_n,x_n)|
    > \epsilon \ , \notag
\end{equation}
with $t_n \to 0$ in the case of Theorem~\ref{maint}, respectively
$t_n \to \infty$ in the case of Theorem~\ref{maints}. In addition,
the flux-energy relation
\[
\F_{[t_1,t_2]}[\Phi] = \E_{S_{t_2}}[\Phi] - \E_{S_{t_1}}[\Phi]
\]
shows that in the case of Theorem~\ref{maint} we have
\[
 \lim_{t_1,t_2 \to 0}  \F_{[t_1,t_2]}[\Phi] = 0
\]
and in the case of Theorem~\ref{maints} we have
\[
 \lim_{t_1,t_2 \to \infty}  \F_{[t_1,t_2]}[\Phi] = 0 \ .
\]
This allows us to also choose $\epsilon_n \to 0$ such that
\begin{equation}
  \F_{[\epsilon_n t_n,t_n]}[\Phi] \leq \epsilon_n^{\frac12}\E[\Phi] \ . \notag
\end{equation}
Rescaling to $t=1$ we produce the sequence of wave maps
\[
 \Phi^{(n)}(t,x) = \Phi( t_n t, t_n x)
\]
in the increasing regions $C_{[\epsilon_n,1]}$ so that
\begin{equation}
 \F_{[\epsilon_n,1]}[\Phi^{(n)}] \leq \epsilon_n^{\frac12} \E[\Phi] \ , \notag
\end{equation}
and also points $x_n \in \R^2$ and frequencies $k_n \in \Z$ so that
\begin{equation}
 | P_{k_n} \Phi^{(n)}(1,x_n)| + 2^{-k_n} |P_{k_n} \partial_t \Phi^{(n)}(1,x_n)|
    > \epsilon \ . \label{bad_disp_bound}
\end{equation}
From this point on, the proofs of Theorems~\ref{maint},\ref{maints}
are identical.\ret


\step{2}{Elimination of null concentration scenario} Using the fixed
time portion of the $X_0$ energy bounds we eliminate the case of
null concentration
\[
 |x_n| \to 1 \ , \qquad k_n \to \infty
\]
in estimate \eqref{bad_disp_bound}, and show that the sequence of
maps $\Phi^{(n)}$ at time $t=1$  must either have low frequency
concentration in the range:
\[
 m(\epsilon,E) <  k_n < M(\epsilon,E) \ ,
\qquad |x_n| < R(\epsilon,E)
\]
or high frequency concentration strictly inside the cone:
\[
 k_n \geq M(\epsilon,E),
\qquad |x_n| < \gamma(\epsilon,E) < 1 \ .
\]\ret


\step{3}{Time-like energy concentration} In both remaining cases
above we show that a nontrivial portion of the energy of
$\Phi^{(n)}$ at time $1$ must be located inside a smaller cone,
\begin{equation}
\frac12\int_{t=1, |x| < \gamma_1 }  \big( |\partial_t \Phi^{(n)}|^2
+ |\nabla_x \Phi^{(n)}|^2\big)\, dx \geq E_1 \notag
\end{equation}
where $E_1 = E_1(\epsilon,E)$ and
$\gamma_1 = \gamma_1(\epsilon,E) < 1$.\ret


\step{4}{Uniform propagation of non-trivial time-like energy} Using
again the $X_0$ energy bounds we propagate the above time-like
energy concentration for $\Phi^{(n)}$ from time $1$ to smaller times
$t \in [\epsilon_n^\frac12,\epsilon_n^\frac14]$,
\begin{equation}
    \frac12\int_{ |x| < \gamma_2(\epsilon,E) t}
     \big( |\partial_t \Phi^{(n)}|^2
    + |\nabla_x \Phi^{(n)}|^2\big)\, dx \geq E_0(\epsilon,E) \ ,
    \qquad t \in [\epsilon_n^\frac12,\epsilon_n^\frac14] \ . \notag
\end{equation}
At the same time, we obtain bounds for $X_0 \Phi^{(n)}$ outside
smaller and smaller neighborhoods of the cone, namely
\begin{equation}
    \int_{C_{[\epsilon_n^\frac12,\epsilon_n^\frac14]}^{\epsilon_n}} \rho^{-1}
    |X_0 \Phi^{(n)}|^2 dx dt \lesssim 1 \ . \notag
\end{equation}\ret


\step{5}{Final rescaling} By a pigeonhole argument and rescaling we
end up producing another sequence of maps, denoted still by
$\Phi^{(n)}$,  which are sections the original wave map $\Phi$ and
are defined in increasing regions $C_{[1,T_n]}$, $T_n = e^{|\ln
\epsilon_n|^\frac12}$, and satisfy the following three properties:
\begin{align}
    \E_{S_t}[\Phi^{(n)}] &\approx E, \ \  \ t \in [1,T_n]&  \text{(Bounded
    Energy)}\notag\\
    \E_{S_t^{(1-\gamma_2)t}}[\Phi^{(n)}] &\geq E_2, \ \  t \in
    [1,T_n]&
     \text{(Nontrivial Time-like Energy)}\notag\\
    \dint_{\!\!C_{[1,T_n]}^{ \epsilon_n^\frac12}} \frac{1}{\rho} |X_0
    \Phi^{(n)}|^2 dx dt &\lesssim  |\log \epsilon_n|^{-\frac12}&
    \text{(Decay to Self-similar Mode)} \notag
\end{align}\ret


\step{6}{Isolating the concentration scales} The compactness result
in Proposition~\ref{pcompact} only applies to wave maps with energy
below the threshold $E_0$ in the small data result. Thus we need to
understand on which scales can such concentration occur. Using
several additional pigeonholing arguments we show that one of the
following two scenarios must occur:

\begin{enumerate}[i)]
\item (Energy Concentration) On a subsequence there exist $(t_n,x_n)
\to (t_0,x_0)$, with $(t_0,x_0)$ inside
$C_{[\frac{1}{2},\infty)}^\frac{1}{2}$, and scales $r_n \to 0$ so that we have
\begin{align}
    \E_{B(x_n,r_n)}[\Phi^{(n)}](t_n) &= \frac{1}{10} E_0 \ ,
    \notag\\
    \E_{B(x,r_n)}[\Phi^{(n)}](t_n) &\leq \frac{1}{10} E_0 \ , \qquad x \in B(x_0,r)
     \ , \notag\\
    r_n^{-1} \int_{t_n-r_n/2}^{t_n+r_n/2} \int_{B(x_0,r)} |X_0
    \Phi^{(n)}|^2 dx dt &\to 0 \ . \notag
\end{align}

\item (Non-concentration) For each $j\in \N$ there exists an $r_j>0$ such that for every
$(t,x)$ inside $C_j=C_{[1,\infty)}^1\cap\{2^j<t<2^{j+1}\}$ one has
\begin{align}
    \E_{B(x,r_j)}[\Phi^{(n)}](t) &\leq \frac{1}{10} E_0 \ , \qquad \forall (t,x) \in C_j \ ,
      \notag\\
    \E_{S_t^{(1-\gamma_2)t}}[\Phi^{(n)}](t) &\geq E_2 \ ,  \notag\\
    \dint_{C_j} |X_0 \Phi^{(n)}|^2 dx dt &\to 0 \ .
    \notag
\end{align}
uniformly in $n$.
\end{enumerate}\ret


\step{7}{The compactness argument} In case i) above we consider the
rescaled wave-maps
\[
 \Psi^{(n)}(t,x) = \Phi^{(n)}(t_n + r_n t, x_n+r_n x)
\]
and show that on a subsequence they converge locally in the energy
norm to a finite energy nontrivial wave map $\Psi$ in $\R^2 \times
[-\frac12,\frac12]$ which satisfies $X(t_0,x_0) \Psi = 0$. Thus
$\Psi$ must be a Lorentz transform of a nontrivial harmonic map.

In case ii) above we show directly that the sequence $\Phi^{(n)}$
converges locally on a subsequence in the energy norm to finite
energy nontrivial wave map $\Psi$, defined in  the interior of a translated cone
$C_{[2,\infty)}^2$, which
satisfies $X_0 \Psi = 0$. Consequently, in hyperbolic coordinates we
may interpret $\Psi$ as a nontrivial harmonic map
\[
\Psi: \H^2 \to \M \ .
\]
Compactifying this and using conformal invariance, we obtain
a non-trivial finite energy harmonic map
\[
\Psi: \mathbb{D}^2 \to \M \
\]
from the unit disk $\mathbb{D}^2$, which according to the  estimates of
Section~\ref{energy_sect} obeys the additional weighted energy bound:
\begin{equation}
    \int_{\mathbb{D}^2} |\nabla_{x}\Phi|^2 \frac{dx}{1-r}
    \ < \ \infty \ . \notag
\end{equation}
But such maps do not exist via combination of a theorem of Qing \cite{MR1223710} and
a theorem of Lemaire \cite{Lem}.\ret


\section{Weighted Energy Estimates for the Wave
  Equation}\label{energy_sect}

In this Section we prove the main energy decay estimates. The
technique we use is the standard one of contracting the energy-momentum tensor:
\begin{equation}
  T_{\alpha\beta}[\Phi] \ = \ m_{ij}(\Phi)\big[ \partial_\alpha\phi^i\partial_\beta\phi^j
  - \frac{1}{2} g_{\alpha\beta}\, \partial^\gamma\phi^i \partial_\gamma\phi^j \big]
  \ , \label{EM_tensor}
\end{equation}
with well chosen vector-fields. Here $\Phi = (\phi^1,\ldots,\phi^n)$
is a set of local coordinates on the target manifold $(\M,m)$ and
$(g_{\alpha\beta})$ stands for the Minkowski metric. The main two
properties of $T_{\alpha\beta}[\Phi] $ are that it is divergence
free $\nabla^\alpha T_{\alpha\beta}=0$, and also that it obeys the
positive energy condition $T(X,Y)\geqslant 0$ whenever both
$g(X,X)\leqslant 0$ and $g(Y,Y)\leqslant 0$. This implies that
contracting $T_{\alpha\beta}[\Phi] $ with timelike/null
vector-fields will result in good energy estimates on characteristic
and space-like hypersurfaces.

If $X$ is some vector-field, we can form its associated momentum
density (i.e. its Noether current)
\begin{equation}
  {}^{(X)}\!\!P_\alpha \ = \ T_{\alpha\beta}[\Phi] X^\beta \ . \notag
\end{equation}
This one form obeys the divergence rule
\begin{equation}
  \nabla^\alpha {}^{(X)}\!\!P_\alpha \ = \ \frac{1}{2} T_{\alpha\beta}[\Phi] {}^{(X)}\!\!\pi^{\alpha\beta}
  \ , \label{divergence}
\end{equation}
where ${}^{(X)}\!\!\pi_{\alpha\beta}$ is the deformation tensor of $X$,
\begin{equation}
  {}^{(X)}\!\!\pi_{\alpha\beta} \ = \ \nabla_\alpha X_\beta + \nabla_\beta X_\alpha \ . \notag
\end{equation}
A simple computation shows that one can also express
\begin{equation}
  {}^{(X)}\!\!\pi \ = \ \mathcal{L}_X g \ . \notag
\end{equation}
 This latter formulation is
very convenient when dealing with coordinate derivatives. Recall that
in general one has:
\begin{equation}
  (\mathcal{L}_X g)_{\alpha\beta} \ = \ X(g_{\alpha\beta}) +
  \partial_\alpha (X^\gamma) g_{\gamma \beta} + \partial_\beta (X^\gamma) g_{\alpha \gamma} \ . \notag
\end{equation}

Our energy estimates are obtained by integrating the relation
\eqref{divergence} over cones $C_{[t_1,t_2]}^\delta$.  Then from
\eqref{divergence} we obtain, for $\delta \leq t_1 \leq t_2$:
\begin{equation}
    \int_{S_{t_2}^\delta}
    {}^{(X)}\!\!P_0 \ dx +
    \frac{1}{2}\dint_{C_{[t_1,t_2]}^\delta}\!\!\!\!\!\!\!
    T_{\alpha\beta}[\Phi] {}^{(X)}\!\!\pi^{\alpha\beta}  \
    dxdt \\
    = \int_{S_{t_1}^\delta}\!\!\!  {}^{(X)}\!\!P_0 \
    dx + \int_{\partial C_{[t_1,t_2]}^\delta}
    \!\!\!\!\!\! {}^{(X)}\!\!P_L \ dA \ ,
    \label{enest}
\end{equation}
where $dA$ is an appropriately normalized (Euclidean) surface area
element on the lateral boundary of the cone $r=t-\delta$.

The standard energy estimates come from contracting
$T_{\alpha\beta}[\Phi]$ with $Y=\partial_t$. Then we have
\[
 {}^{(Y)}\!\!\pi = 0 \ , \qquad    {}^{(Y)}\!\!P_0 = \frac12(|\partial_t
 \Phi|^2 + |\nabla_x \Phi|^2) \ , \qquad   {}^{(Y)}\!\!P_L = \frac14 |L
 \Phi|^2 + \frac12 |\spartial \Phi|^2 \ .
\]
Applying \eqref{enest} over $C_{[t_1,t_2]}$ we obtain the energy-flux relation
\eqref{en-flux} used in the introduction.
Applying \eqref{enest} over $C_{[\delta,1]}^\delta$ yields
\begin{equation}
    \int_{\partial C_{[\delta,1]}^\delta} \frac14 |L \Phi|^2 +
    \frac12 |\spartial \Phi|^2  \ dA \leq \E[\Phi] \ .
    \label{enyd}
\end{equation}

It will also be necessary for us to have a version of the usual energy estimate adapted to
the hyperboloids $\rho=\sqrt{t^2-r^2}=const$. Integrating the divergence of
the ${}^{(Y)}\!\!P_\alpha$ momentum density over regions of the form
$\mathcal{R}=\{\rho\geqslant \rho_0,t\leqslant t_0\}$ we have:
\begin{equation}
    \int_{\{\rho=\rho_0\}\cap\{t\leqslant t_0\}}  {}^{(Y)}\!\!P^\alpha  dV_\alpha \leq \E[\Phi] \ ,
    \label{hyp_eng}
\end{equation}
where the integrand on the LHS denotes the interior product of ${}^{(Y)}\!\!P$
with the Minkowski volume element. To express this estimate in a useful way,
we use the hyperbolic coordinates (CMC foliation):
\begin{align}
  t \ &= \ \rho \cosh(y) \ , &r \ &= \ \rho \sinh(y) \ , &\theta \ &=
  \ \Theta \ . \ \label{hyp_coords}
\end{align}
In this system of coordinates, the Minkowski metric becomes
\begin{equation}
  -dt^2 + dr^2 + r^2d\theta^2  \ = \ -d\rho^2 +
  \rho^2\big( dy^2 + \sinh^2(y)d\Theta^2 \big) \ . \label{metric}
\end{equation}
A quick calculation shows that the contraction on line
\eqref{hyp_eng} becomes the one-form
\begin{equation}
    {}^{(Y)}\!\!P^\alpha  dV_\alpha \ = \ T(\partial_\rho,\partial_t)\rho^2 dA_{\mathbb{H}^2} \ ,
    \qquad dA_{\mathbb{H}^2} \ =\ \sinh(y)dy d\Theta \ .
\end{equation}
The area element $dA_{\mathbb{H}^2}$ is that of the hyperbolic plane $\mathbb{H}^2$.
To continue, we note that:
\begin{equation}
    \partial_t \ = \ \frac{t}{\rho}\partial_\rho - \frac{r}{\rho^2}\partial_y \ , \notag
\end{equation}
so in particular
\begin{equation}
    T(\partial_\rho,\partial_t) \ = \ \frac{\cosh(y)}{2}|\partial_\rho\Phi|^2
    - \frac{\sinh(y)}{\rho}\partial_\rho\Phi\cdot\partial_y\Phi
    + \frac{\cosh(y)}{2\rho^2}\Big(
    |\partial_y\Phi|^2 + \frac{1}{\sinh^2(y)}|\partial_\Theta\Phi|^2
    \Big)   \ . \notag
\end{equation}
Letting $t_0 \to \infty$ in \eqref{hyp_eng} we obtain a useful
consequence of this, namely a weighted hyperbolic space estimate for
special solutions to the wave-map equations, which will be used in the
sequel to rule out the existence of non-trivial finite energy
self-similar solutions:

\begin{lem}\label{hyp_eng_lem}
Let $\Phi$ be a finite energy smooth wave-map in the interior of the cone $C$. 
Assume also that $\partial_\rho\Phi\equiv 0$.
Then one has:
\begin{equation}
      \frac{1}{2}\int_{\mathbb{H}^2}|\nabla_{\mathbb{H}^2}\Phi|^2 \cosh(y)
    dA_{\mathbb{H}^2}
    \ \leq \ \mathcal{E}[\Phi] \ . \label{drho0_hyp_eng}
\end{equation}
Here:
\begin{equation}
    |\nabla_{\mathbb{H}^2}\Phi|^2 \ =\  |\partial_y\Phi|^2 + \frac{1}{\sinh^2(y)}|\partial_\Theta\Phi|^2
    \ , \notag
\end{equation}
is the covariant energy density for the hyperbolic metric.
\end{lem}\ret

Our next order of business is to obtain decay estimates for
time-like components of the energy density. For this we  use the
timelike/null vector-field
\begin{equation}
  X_\epsilon \ = \ \frac{1}{\rho_\epsilon}\left((t+\epsilon)\partial_t
    + r\partial_r \right), \qquad \rho_\epsilon = \sqrt{(t+\epsilon)^2-r^2} \ .
 \label{X_field}
\end{equation}
In order to gain some intuition, we first consider the case of
$X_0$. This is most readily expressed in the system of hyperbolic coordinates
\eqref{hyp_coords} introduced above.
One  easily checks that the coordinate derivatives turn out to be
\begin{align}
  \partial_\rho \ &= \ X_0 \ , &\partial_y \ = \ r\partial_t +
  t\partial_r \ . \notag
\end{align}
In particular, $X_0$ is uniformly timelike with $g(X_0,X_0)=-1$, and
one should expect it to generate good energy estimate on time slices
$t=const$.    In
the system of coordinates \eqref{hyp_coords} one also has that
\begin{align}
  \mathcal{L}_{X_0} g \ = \ 2\rho \big( dy^2 + \sinh^2(y)d\Theta^2
  \big) \ . \notag
\end{align}
Raising this, one then computes
\begin{equation}
  {}^{(X_0)}\!\!\pi^{\alpha\beta} \ = \
  \frac{2}{\rho^3} \big( \partial_y\otimes\partial_y +
  \sinh^{-2}(y)\partial_\Theta\otimes\partial_\Theta \big) \ . \notag
\end{equation}
Therefore, we have the contraction identity:
\begin{equation}
  \frac{1}{2} T_{\alpha\beta}[\Phi] {}^{(X_0)}\!\!\pi^{\alpha\beta} \ = \
  \frac{1}{\rho} |X_0 \Phi|^2 \ . \notag
\end{equation}
To compute the components of ${}^{(X_0)}\!\!P_0$ and
${}^{(X_0)}\!\!P_L$ we use the associated optical functions
\begin{align}
  u \ &= \ t - r \ , &v \ &= \
  t + r \ . \notag
\end{align}
Notice that $\rho^2 = u v$. Also, simple
calculations show that
\begin{align}
  X_0 \ &= \ \frac{1}{\rho} \left(
    \frac{1}{2}v L + \frac{1}{2}u \bL \right) \ ,
  &\partial_t \ = \ \frac{1}{2} L + \frac{1}{2} \bL \ .
  \label{vect_decomp}
\end{align}
Finally, we record here the components of $T_{\alpha\beta}[\Phi]$ in
the null frame
\begin{align}
  T(L,L) \ &= \ |L\Phi|^2 \ , &T(\bL,\bL) \ &= \ |\bL\Phi|^2 \ ,
  &T(L,\bL) \ &= \ |\spartial \Phi|^2 \ . \notag
\end{align}
By combining the above calculations, we see that we may compute
\begin{align}
  {}^{(X_0)}\!\!P_0  &=  T(\partial_t , X_0)  =
  \frac{1}{4} \left(\frac{v}{u}\right)^\frac{1}{2}
  |L\Phi|^2 + \frac{1}{4}\Big[
  \left(\frac{v}{u}\right)^\frac{1}{2} +
  \left(\frac{u}{v}\right)^\frac{1}{2} \Big]
  |\spartial \Phi|^2 + \frac{1}{4}
  \left(\frac{u}{v}\right)^\frac{1}{2}
  |\bL\Phi|^2 \ , \notag \\
  {}^{(X_0)}\!\!P_L  &=  T(L , X_0)  =  \frac{1}{2}
  \left(\frac{v}{u}\right)^\frac{1}{2} |L\Phi|^2 +
  \frac{1}{2} \left(\frac{u}{v}\right)^\frac{1}{2}
  |\spartial \Phi|^2 \ . \notag
\end{align}
These are essentially the same as the components of the usual energy
currents ${}^{(\partial_t)}\!\!P_0 $ and ${}^{(\partial_t)}\!\!P_L $
modulo ratios of the optical functions $u$ and $v$.

One would expect to get nice space-time estimates for
$X_0\Phi$ by integrating \eqref{divergence} over the interior cone
$r\leqslant t\leqslant 1$. The only problem  is that the
boundary terms  degenerate rather severely when $\rho\to 0$.  To
avoid this we simply redo everything with the shifted version
$X_\epsilon$ from line \eqref{X_field}. The above formulas
remain valid with $u$, $v$ replaced by their time shifted versions
\begin{align}
  u_\epsilon \ &= \ (t+\epsilon) - r \ , &v_\epsilon \ &= \
  (t+\epsilon) + r \ . \notag
\end{align}

Furthermore, notice that for small $t$ one has in the region $r\leqslant t$ the bounds
\begin{align}
    \left(\frac{v_\epsilon}{u_\epsilon}\right)^\frac{1}{2} \approx 1 \ ,
    \qquad \left(\frac{u_\epsilon}{v_\epsilon}\right)^\frac{1}{2} \ \approx 1 \ ,
    \qquad 0 <t \leq \epsilon \ .
    \notag
\end{align}
Therefore, one has in $r\leqslant t$ that
\begin{equation}
  {}^{(X_\epsilon)}\!\!P_0 \ \approx
  {}^{(\partial_t)}\!\!P_0 \ , \qquad  0< t \leq \epsilon \notag \ .
\label{jcomp}\end{equation}

In what follows we work with a wave-map $\Phi$ in $C_{[\epsilon,1]}$.
We denote its total energy and flux by
\[
E = \E_{S_1}[\Phi] \ , \qquad F = \F_{[\epsilon,1]}[\Phi] \ .
\]
In the limiting case $F=0$, $\epsilon = 0$ one could apply
\eqref{enest} to obtain
\[
  \int_{S_{t_2}^0}
  {}^{(X_\epsilon)}\!\!P_0 \ dx + \dint_{C_{[t_1,t_2]}^0}\
  \frac{1}{\rho_\epsilon} |X_\epsilon \Phi|^2\
  dxdt \\
  = \int_{S_{t_1}^0}\  {}^{(X_\epsilon)}\!\!P_0 \
  dx \ .
\]
By \eqref{jcomp}, letting  $t_1 \to 0$ followed by $\epsilon \to 0$
and taking supremum over $0 < t_2\leqslant 1$
we would get the model estimate
\[
\sup_{t \in (0,1]}  \int_{S_{t}^0}
  {}^{(X_0)}\!\!P_0 \ dx + \dint_{C_{[0,1]}^0}\
  \frac{1}{\rho} |X_0 \Phi|^2\
  dxdt \leq E \ .
\]

However, here we need to deal with a small nonzero flux. Observing
that
\[
{}^{(X_\epsilon)}\!\!P_L \ \lesssim \
  \epsilon^{-\frac{1}{2}}  {}^{(\partial_t)}\!\!P_L \ ,
\]
from  \eqref{enest} we obtain the weaker bound
\[
 \int_{S_{t_2}^0}
  {}^{(X_\epsilon)}\!\!P_0 \ dx + \dint_{C_{[t_1,t_2]}^0}\
  \frac{1}{\rho_\epsilon} |X_\epsilon \Phi|^2\
  dxdt \\
  \lesssim  \int_{S_{t_1}^0}\  {}^{(X_\epsilon)}\!\!P_0 \
  dx  +  \epsilon^{-\frac{1}{2}} F \ .
\]
Letting $t_1 = \epsilon$ and taking supremum over $\epsilon\leqslant t_2\leqslant 1$
we obtain
\begin{equation}
\sup_{t \in (\epsilon,1]}  \int_{S_{t}^0}
  {}^{(X_\epsilon)}\!\!P_0 \ dx + \dint_{C_{[0,1]}^0}\
  \frac{1}{\rho_\epsilon} |X_\epsilon \Phi|^2\
  dxdt \lesssim  E + \epsilon^{-\frac{1}{2}} F
\label{enxe}\end{equation}
A consequence of this is the following, which will be used to rule out
the case of asymptotically  null pockets of energy:

\begin{lem}
Let $\Phi$ be a smooth wave-map in the cone $C_{(\epsilon,1]}$
which satisfies the flux-energy relation $F \lesssim \epsilon^\frac12 E$. Then
\begin{equation}
  \int_{S_{1}^0}  {}^{(X_\epsilon)}\!\!P_0 \ dx \lesssim E \ .
\label{enxea}\end{equation}
\end{lem}\ret

Next, we show can replace $X_\epsilon$ by $X_0$ in \eqref{enxe} if we restrict
the integrals on the left to $r < t-\epsilon$. In this region we
have
\[
   {}^{(X_\epsilon)}\!\!P_0 \approx   {}^{(X_0)}\!\!P_0 \ ,
\qquad \rho_\epsilon \approx \rho \ .
\]
Using second member above, a direct computation shows that in $r < t-\epsilon$
\[
    \frac{1}{\rho} |X_0 \Phi|^2
    \lesssim  \frac{1}{\rho_\epsilon} |X_\epsilon \Phi|^2
    + \frac{\epsilon^2}{\rho^3} |\partial_t \Phi|^2
\]
and also
\[
    \int_{C^\epsilon_{(\epsilon,1]}}  \frac{\epsilon^2}{\rho^3} |\partial_t
    \Phi|^2 dx dt \leq \int_{C^\epsilon_{(\epsilon,1]} }
    \frac{\epsilon^\frac12}{t^\frac32} |\partial_t  \Phi|^2 dx dt \lesssim E \ .
\]
Thus, we have proved the following estimate which will be used to
conclude that rescaling of $\Phi$ are asymptotically stationary, and
also used to help trap uniformly time-like pockets of energy:

\begin{lem}
Let $\Phi$ be a smooth wave-map in the cone $C_{(\epsilon,1]}$ which satisfies the flux-energy relation $F \lesssim \epsilon^\frac12 E$.
Then we have
\begin{equation}
\sup_{t \in (\epsilon,1]}  \int_{S_{t}^\epsilon}
  {}^{(X_0)}\!\!P_0 \ dx + \dint_{C_{[\epsilon,1]}^\epsilon}\
  \frac{1}{\rho} |X_0 \Phi|^2
  dxdt \lesssim  E \ .
\label{xphi}\end{equation}
\end{lem}\ret

Finally, we use the last lemma to propagate pockets of energy forward away from the
boundary of the cone. By \eqref{enest} for $X_0$ we have
\[
\int_{S_{1}^\delta}  {}^{(X_0)}\!\!P_0 \ dx \leq
\int_{S_{t_0}^\delta}
 {}^{(X_0)}\!\!P_0 \ dx + \int_{\partial C_{[t_0,1]}^\delta}\
  {}^{(X_0)}\!\!P_L \ dA \ , \qquad  \epsilon \leq \delta < t_0 < 1 \ .
\]
We consider the two components of $ {}^{(X_0)}\!\!P_L$ separately. For the
angular component, by \eqref{enyd} we have the bound
\[
 \int_{\partial C_{[t_0,1]}^\delta}\  \left(\frac{u}{v}\right)^\frac{1}{2}
  |\spartial \Phi|^2 \ dA \lesssim \left(\frac{\delta}{t_0}\right)^\frac{1}{2}
\int_{\partial C_{[t_0,1]}^\delta}\
  |\spartial \Phi|^2 \ dA \lesssim \left(\frac{\delta}{t_0}\right)^\frac{1}{2} E \ .
\]
For the $L$ component a direct computation shows that
\[
   |L \Phi|  \lesssim \left(\frac{u}{v}\right)^\frac{1}{2} |X_0 \Phi|
+ \left(\frac{u}{v}\right) | \bL \Phi| \ .
\]
Thus we obtain
\[
\int_{S_{1}^\delta}  {}^{(X_0)}\!\!P_0 \ dx
\lesssim\int_{S_{t_0}^\delta}\!\!\!\!\!
 {}^{(X_0)}\!\!P_0 \ dx +
\left(\frac{\delta}{t_0}\right)^\frac{1}{2}\!\! E
 +    \int_{\partial C_{[t_0,1]}^\delta}\!\!\Big(
 \left(\frac{u}{v}\right)^\frac{1}{2}  |X_0 \phi|^2 +
 \left(\frac{u}{v}\right)^\frac{3}{2} | \bL \Phi|^2\Big)
   dA \ .
\]
For the last term we optimize with respect to $\delta \in [\delta_0,\delta_1]$ to obtain:

\begin{lem}
  Let $\Phi$ be a smooth wave-map in the cone $C_{(\epsilon,1]}$ which
  satisfies the flux-energy relation $F \lesssim \epsilon^\frac12
  E$. Suppose that $ \epsilon \leq \delta_0 \ll \delta_1 \leq t_0$.
  Then
\begin{equation}
    \int_{S_{1}^{\delta_1}}  {}^{(X_0)}\!\!P_0 \ dx \lesssim
    \int_{S_{t_0}^{\delta_0}}  {}^{(X_0)}\!\!P_0 \ dx + \left(
    \left(\frac{\delta_1}{t_0}\right)^\frac{1}{2}+
    (\ln(\delta_1/\delta_0))^{-1}\right) E \ .
    \label{energyup}
\end{equation}
\label{penergyup}
\end{lem}

To prove this lemma, it suffices to choose
$\delta \in [\delta_0,\delta_1]$ so that
\[
 \int_{\partial C_{[t_0,1]}^\delta}\ \left[
 \left(\frac{u}{v}\right)^\frac{1}{2}  |X_0 \phi|^2 +
 \left(\frac{u}{v}\right)^\frac{3}{2} | \bL \Phi|^2\right]
  \ dA \lesssim |\ln(\delta_1/\delta_0)|^{-1} E \ .
\]
This follows by pigeonholing  the estimate
\[
 \int_{ C_{[t_0,1]}^{\delta_0} \setminus  C_{[t_0,1]}^{\delta_1} }\
\frac{1}{u}\left[ \left(\frac{u}{v}\right)^\frac{1}{2}  |X_0 \phi|^2 +
 \left(\frac{u}{v}\right)^\frac{3}{2} | \bL \Phi|^2 \right] dx dt
\lesssim E \ .
\]
The first term is estimated directly by \eqref{xphi}.  For the second
we simply use energy bounds since  in the domain of integration
we have the relation
\[
\frac{1}{u} \left(\frac{u}{v}\right)^\frac{3}{2} \leq
\frac{\delta_1^\frac12}{t^\frac32} \ .
\]\ret


\section{Finite Energy Self Similar Wave-Maps}\label{SS_sect}
The purpose of this section is to prove the following Theorem:

\begin{thm}[Absence of non-trivial finite energy self similar wave-maps in $2D$]\label{ss_thm}
Let $\Phi$ be a finite energy solution to the wave-map equation \eqref{main_eq}
defined in the forward cone $C$.
Suppose also
that $\partial_\rho\Phi\equiv 0$.
Then $\Phi\equiv const$.
\end{thm}

\begin{rem}
Theorems of this type are standard in the literature. The first such
reference we are aware of is in the work of Shatah--Tahvildar-Zadeh
\cite{MR96c:58049} on the equivariant case. This was later extended
by Shatah-Struwe \cite{MR1674843} to disprove the existence of
(initially) \emph{smooth} self-similar profiles. However, the
authors are not aware of an explicit reference in the literature
ruling out the  possibility of general (i.e. non-symmetric) finite
energy self-similar solutions to the system \eqref{main_eq},
although this statement has by now acquired the status of a
folk-lore theorem. It is this latter version
in the above form that is necessary in the context of the present work.
\end{rem}

\begin{rem}
  It is important to remark that the finite energy assumption cannot
  be dropped, and also that this failure is not due to interior
  regularity. That is, there are non-trivial $C^\infty$ self similar
  solutions to \eqref{main_eq} in $C$ but these solutions all have
  infinite energy. However, the energy divergence is marginal,
  i.e. the energy in $S_t^\delta$ only grows as $|\ln(\delta)|$ as
  $\delta \to 0$. Further, these solutions have finite energy
when viewed as harmonic maps in $\H^2$. 
\end{rem}\ret

\begin{proof}[Proof of Theorem \ref{ss_thm}]
Writing out the system \eqref{main_eq} in the coordinates
\eqref{hyp_coords}, and canceling the balanced factor of $\rho^{-2}$
on both sides we have that $\Phi$ obeys:
\begin{equation}
    \Delta_{\mathbb{H}^2}\Phi^a \ = \ - g^{ij}_{\mathbb{H}^2}\mathcal{S}^a_{bc}(\Phi)\partial_i\Phi^b\partial_j\Phi^c
    \ , \notag
\end{equation}
where in polar coordinates $g_{\mathbb{H}^2}=dy^2+\sinh^2(y)d\Theta^2$
is the standard hyperbolic metric. Thus, $\Phi^a$ is an entire
harmonic map $\Phi:\mathbb{H}^2\to \mathcal{M}$. By elliptic
regularity, such a map is smooth with uniform bounds on compact sets
(see \cite{MR1913803}). In particular, $\Phi$ is $C^\infty$ in the
interior of the cone $C$.  Therefore, from estimate
\eqref{drho0_hyp_eng} of Lemma \ref{hyp_eng_lem}, $\Phi$ enjoys the
additional weighted energy estimate:
\begin{equation}
    \int_{\mathbb{H}^2} |\nabla_{\mathbb{H}^2}\Phi|^2 \cosh(y)
    dA_{\mathbb{H}^2} \ \leq \ 2E(\Phi) \ = \
    \int_{S_t} \big(|\partial_t\Phi|^2 + |\nabla_x\Phi|^2 \big)
    dx
    \ , \label{nat_hyp_eng}
\end{equation}
for any fixed $t>0$.

To proceed further, it is  convenient to rephrase all of the above
in terms of the conformal compactification of $\mathbb{H}^2$. Using the pseudo-spherical
stereographic projection from hyperboloids to the unit disk $\mathbb{D}^2=\{t=0\}\cap\{x^2+y^2< 1\}$
in Minkowski space  (see Chapter II of \cite{MR1083149}):
\begin{equation}
    \pi(t,x^1,x^2) \ = \ (-1,0,0) - \frac{2\rho(t+\rho,x^1,x^2) }
{\big\langle(t+\rho,x^1,x^2),(t+\rho,x^1,x^2)\big\rangle} \ , \notag
\end{equation}
as well as the conformal invariance of the $2D$ harmonic
map equation and its associated Dirichlet energy,
we have from \eqref{nat_hyp_eng} that $\Phi$ induces a finite energy
harmonic map  $\Phi:\mathbb{D}^2\to\mathcal{M}$ with the additional property that:
\begin{equation}
    \frac{1}{2}\int_{\mathbb{D}^2} |\nabla_{x}\Phi|^2 \Big(\frac{1+r^2}{1-r^2}\Big) dx \
    < \ \infty  \ . \label{D_nat_hyp_eng}
\end{equation}

To conclude, we only need to show that such maps are trivial. Notice
that the weight $(1-r)^{-1}$ is \emph{critical} in this respect, for
there are many non-trivial\footnote{For example, if $\mathcal{M}$ is
a complex surface with conformal metric $m=\lambda^2 d\Phi
d\overline{\Phi}$, then the harmonic map equation becomes
$\partial_z\partial_{\overline{z}} \Phi= -2\partial_\Phi(\ln\lambda)
\partial_z\Phi\partial_{\overline{z}}\Phi$. Thus, any holomorphic
function $\Phi:\mathbb{D}^2\to\mathcal{M}$ suffices to produce an
infinite energy self similar ``blowup profile'' for \eqref{main_eq}.
See Chapter I of \cite{MR1474501}.} finite energy harmonic maps from
$\mathbb{D}^2\to\mathcal{M}$, regardless of the curvature of
$\mathcal{M}$, which are also uniformly smooth up to the boundary
$\partial\mathbb{D}^2$ and may therefore absorb an energy weight of
the form $(1-r)^{-\alpha}$ with $\alpha<1$.

From the bound \eqref{D_nat_hyp_eng}, we have that there exists a sequence of radii $r_n\nearrow 1$ such that:
\begin{equation}
    \int_{r=r_n} |\snabla \Phi|^2 dl \ = \ o_n(1) \ . \notag
\end{equation}
From the uniform boundedness of $\Phi(r)$ in $L^2(\mathbb{S}^1)$ and
the trace theorem, this implies that:
\begin{equation}
    \lp{\Phi|_{\partial \mathbb{D}^2}}{L^2(\mathbb{S}^1)} \ < \ \infty \ , \qquad
    \lp{\Phi|_{\partial \mathbb{D}^2}}{\dot{H}^\frac{1}{2}(\mathbb{S}^1)} \ = \ 0 \ , \notag
\end{equation}
so therefore $\Phi|_{\partial \mathbb{D}^2}\equiv const$. In
particular, $\Phi:\mathbb{D}^2\to\mathcal{M}$ is a finite energy
harmonic map  with smooth boundary values. By a theorem of J. Qing
(see \cite{MR1223710}), it follows that $\Phi$ has uniform
regularity up to $\partial \mathbb{D}^2$. Thus, $\Phi$ is $C^\infty$
on the closure $\overline{\mathbb{D}^2}$ with constant boundary
value. By Lemaire's uniqueness theorem (see Theorem 8.2.3 of
\cite{MR1351009}, and originally \cite{Lem}) $\Phi\equiv const$
throughout $\mathbb{D}^2$. By inspection of the coordinates
\eqref{hyp_coords} and the fact that $\partial_\rho\Phi\equiv 0$,
this easily implies that $\Phi$ is trivial in all of
$C$.
\end{proof}

\ret

\section{A Simple Compactness Result}\label{X_sect}

The main aim of this section is the following result:

\begin{prop}
 Let $Q$ be the unit cube, and  let $\Phi^{(n)}$ be a family of wave maps in  $3Q$
  which have small energy $\E[\Phi^{(n)}] \leq \frac{1}{10}E_0$ and so that
  \begin{equation}
    \|X \Phi^{(n)}\|_{L^2(3Q)} \to 0
  \label{time_decay}\end{equation}
for some time-like vector field $X$. Then there exists a
wave map $\Phi \in H^{\frac32-\epsilon}(Q)$ with $E(\phi) \leq\frac{1}{10}E_0 $ so that on a
subsequence we have the strong convergence
  \[
  \Phi^{(n)} \to \Phi \qquad \text { in } H^1(Q) \ .
  \]
\label{pcompact}
\end{prop}

\begin{proof}
  The argument we present here is inspired by the one in Struwe's work
  on Harmonic-Map heat flow (see \cite{MR826871}). The main point there is
  that compactness is gained through the higher regularity afforded
  via integration in time. For a parabolic equation, integrating $L^2$
  in time actually gains a whole derivative by scaling, so in the heat
  flow case one can control a quantity of the form $\int\int
  |\nabla^2\Phi|^2 dxdt$ which leads to control of $\int
  |\nabla^2\Phi|^2 dx$ for many individual points in time.


  By the small data result in \cite{Tataru_WM1}, we have a similar (uniform) space-time bound
  $\Phi^{(n)}$ in any compact subset $K$ of $3Q$ which gains us $\frac{1}{2}$
  a derivative over energy, namely
  \begin{equation}
    \lp{\chi \Phi^{(n)}}{X^{1,\frac{1}{2}}_\infty} \
    \lesssim 1 \ , \qquad \text{supp } \chi \subset 3Q \label{X_bound}
  \end{equation}
  where $X^{1,\frac{1}{2}}_\infty$ here denotes the $\ell^\infty$
  Besov version of the critical inhomogeneous $X^{s,b}$ space.

  We obtain a strongly converging ${H}^1_{t,x}$ sequence of
  wave-maps through a simple frequency decomposition argument as
  follows. The vector field $X$ is timelike, therefore
  its symbol is elliptic in a region of the form $ \{\tau > (1-2\delta) |\xi|\}$
  with $\delta > 0$. Hence, given a cutoff function $\chi$ supported
  in $3Q$ and with symbol equal to $1$ in $2Q$, there exists a microlocal
  space-time decomposition
  \[
  \chi = Q_{-1}(D_{x,t},x,t) X + Q_0(D_{x,t},x,t) + R(D_{x,t},x,t)
  \]
  where the symbols $q_{-1} \in S^{-1}$ and $q_0 \in S^0$ are supported
  in $ 3Q \times \{\tau > (1-2\delta) |\xi|\}$, respectively $ 3Q\times \{\tau < (1-\delta) |\xi|\}$,
  while the remainder $R$ has symbol $r \in S^{-\infty}$ with spatial support in $3Q$.
 This yields  a decomposition for $\chi \Phi^{(n)}$,
\begin{equation}
   \chi \Phi^{(n)} \ = \  (Q_0(D,x) + R(D,x)) \Phi^{(n)}+
     Q_{-1}(D,x) X \Phi^{(n)}
    \ = \ \Phi^{(n)}_{bulk} + \mathcal{R}_n \ . \label{micro_decomp}
\end{equation}
Due to the support properties of $q_0$, for the main term
we have the bound
\[
\| \Phi^{(n)}_{bulk} \|_{{H}^{\frac{3}{2}-\epsilon}_{t,x}}
\lesssim \lp{\chi \Phi^{(n)}}{X^{1,\frac{1}{2}}_\infty} \lesssim 1 \ .
\]
On the other hand the remainder decays in norm,
\[
\lp{\mathcal{R}_n}{{H}^1_{t,x}}{} \lesssim \|X \Phi^{(n)}\|_{L^2(3Q)}
\ \longrightarrow_{n\to \infty} \ 0 \ .
\]
Hence on a subsequence we have the strong convergence
\[
\chi \Phi^{(n)} \to \Phi \qquad \text{in} \ H^1_{t,x}(3Q) \ .
\]
In addition, $\Phi$ must satisfy both
\[
\Phi \in H^{\frac32-\epsilon} \ , \qquad X \Phi = 0 \  \ \text{in\ \ } 2Q \ .
\]
It remains to show that $\Phi$ is a wave-map in $Q$ (in fact a ``strong'' finite energy wave-map
according to the definition of Theorem \ref{tsmall}). There exists a time section
$2Q_{t_0}$ close to the center of $2Q$ such that both
\[
    \lp{\Phi[t_0]}{(\dot{H}^1\times L^2)(2Q_{t_0})} < \infty \ , \qquad
    \lp{\Phi^{(n)}[t_0]-\Phi[t_0]}{(\dot{H}^s\times \dot{H}^{s-1})(2Q_{t_0})} \to 0 \ ,
\]
for some $s<1$.
Letting $\td{\Phi}$ be the solution to \eqref{main_eq} with data
$\Phi[t_0]$, from the weak stability result \eqref{weaklip} in Theorem~\ref{tsmall}
we have $\Phi^{(n)} \to \td{\Phi}$ in $H^s_x(Q_t)$ at fixed time for $s<1$. Thus
 $\td{\Phi}=\Phi$ in $H^s_{t,x}(Q)$ which suffices.
\end{proof}

We consider now the two cases we are interested in, namely when $X = \partial_t$
or $X = t \partial_t + x \partial_x$. If $X = \partial_t$
(as will be the case for a general time-like $X$ vector after boosting), then $\Phi$ is a harmonic map
\[
\Phi: Q \to \M
\]
and is therefore smooth (see \cite{MR1913803}).

If $X = t \partial_t + x \partial_x$
and $3Q$ is contained within the cone $\{ t > |x|\}$
then $\Phi$ can be interpreted as a portion of a self-similar Wave-Map, and
therefore it is  a harmonic map from a domain
\[
\Phi: \H^2\supseteq \Omega  \to \M
\]
and is again smooth. Note that since the Harmonic-Map equation is conformally invariant,
one could as well interpret this as a special case of the previous one. However, in the situation
where we have similar convergence on a large number of such domains $\Omega$ that fill up
$\H^2$,  $\Phi$ will be globally defined as an $\H^2$ harmonic map to $\M$, and we will
therefore be in a position to apply Theorem \ref{ss_thm}.\ret


\section{Proof of Theorems~\ref{maint},\ref{maints}}
\label{theproof}
We proceed in a series of steps:


\subsection{Extension and scaling in the blowup scenario}
We begin with Theorem~\ref{maint}. Let $\Phi$ be a wave map in $C_{(0,1]}$
with terminal energy
\[
 E = \lim_{t \to 0} \E_{S_t}[\Phi] \ .
\]
Suppose that the energy dispersion
scenario B) does not apply. Let $\epsilon > 0$ be so that B) does
not hold. We can choose $\epsilon$ arbitrarily small.
We will take advantage of this  to construct an
extension of $\Phi$ outside the cone which satisfies \eqref{caseben},
therefore violating \eqref{casebed} on any time interval $(0,t_0]$.

For this we use energy estimates. Setting
\[
\F_{t_0} [\Phi] =   \int_{\partial C_{(0,t_0]}}
 \big( \frac{1}{4}|L \Phi|^2 +  \frac{1}{2}|\spartial \Phi|^2\big) dA
\]
it follows that as $t_0 \to 0$ we have
\begin{equation}
 \F_{t_0} [\Phi] =\E_{S_{t_0}}[\Phi] -  E \to 0 \ .
\label{fluxdecay}\end{equation}
Then by pigeonholing we can choose $t_0$ arbitrarily small
so that we have the bounds
\begin{equation}
 \F_{t_0} [\Phi] \ll \epsilon^8 E, \qquad \int_{\partial S_{t_0}}
 |\spartial \Phi|^2 ds \ll \frac{\epsilon^8}{t_0} E \ .
\label{pige}\end{equation}
The second bound allows us to extend the initial data for $\Phi$
at time $t_0$ from $S_{t_0}$ to all of $\R^2$ in such a way that
\[
\E[\Phi](t_0) -  \E_{S_{t_0}}[\Phi] \ll \epsilon^8 E \ .
\]
We remark that by scaling it suffices to consider the case $t_0 = 1$.
The second bound in \eqref{pige} shows, by integration, that the
range of $\Phi$ restricted to $\partial S_{t_0}$ is contained
in a small ball of size $\epsilon^8$ in $\M$. Thus the extension
problem is purely local in $\M$, and can be carried out in a suitable
local chart by a variety of methods.

We extend the solution $\Phi$ outside the cone $C$ between
times $t_0$ and $0$ by solving the wave-map equation.
By energy estimates it follows that for $t \in (0,t_0]$ we have
\begin{equation}
\E[\Phi](t) -  \E_{S_{t}}[\Phi] = \E[\Phi](t_0) - \E_{S_{t_0}}[\Phi]
+ \F_{t_0}[\Phi] - \F_{t}[\Phi] \leq \frac{1}{2}\epsilon^8 E \ .
\label{noeout}
\end{equation}
Hence the energy stays small outside
the cone, and by the small data result in Theorem~\ref{tsmall} there
is no blow-up can occur outside the cone up to time $0$.  The
extension we have constructed is fixed for the rest of the proof.

Since our extension satisfies \eqref{caseben} but B) does not hold,
it follows that we can find a sequence $(t_n,x_n)$ with $t_n \to 0$  and $k_n \in \Z$ so that
\begin{equation}
 | P_{k_n} \Phi(t_n,x_n)| + 2^{-k}
|P_{k_n} \partial_t \Phi(t_n,x_n)| \geq \epsilon \ .
\label{pconc}\end{equation}
The relation \eqref{fluxdecay} shows that $\F_{t_n} \to 0$.
Hence we can find a sequence $\epsilon_n \to 0$
such that $ \F_{t_n} < \epsilon_n^\frac12 E$. Thus
\begin{equation}
\F_{[\epsilon_n t_n,t_n]} < \epsilon_n^\frac12 E \ .
\label{fluxtn}\end{equation}
Rescaling we obtain the sequence of wave maps
\[
 \Phi^{(n)}(t,x) = \Phi( t_n  t, t_n  x)
\]
in the increasing family of regions $ \R^2 \times [\epsilon_n,1]$
with the following properties:\ret

a) Uniform energy size,
\begin{equation}
 \E[\Phi^{(n)}] \approx E \ .
\label{entnresc} \end{equation}

b) Small energy outside the cone,
\begin{equation}
\E[\Phi^{(n)}] - \E_{S_t}[\Phi^{(n)}] \lesssim \epsilon^8 E \ .
\label{outentnresc} \end{equation}

c) Decaying flux,
\begin{equation}
\F_{[\epsilon_n,1]}[\Phi^{(n)}] < \epsilon_n^\frac12 E \ .
\label{fluxtnresc}\end{equation}

d) Pointwise concentration at time $1$,
\begin{equation}
 | P_{k_n} \Phi^{(n)}(1,x_n)| + 2^{-k_n} |P_{k_n} \partial_t\Phi^{(n)} (1,x_n)| > \epsilon \ .
\label{pointnresc}\end{equation}
for some $x_n \in \R^2$, $k_n \in \Z$.\ret


\subsection{Extension and scaling in the non-scattering scenario}

Next we consider the case of Theorem~\ref{maints}.  Again we suppose
that the energy dispersion scenario B) does not apply for a finite energy wave map
$\Phi$ in $C_{[1,\infty)}$. Let $\epsilon > 0$ be so that B) does
not hold. Setting
\[
E = \lim_{t_0 \to \infty} \E_{S_{t_0}}[\Phi] \ , \qquad
\F_{t_0} [\Phi] =   \int_{\partial C_{[t_0,\infty)}}
 \big( \frac{1}{4}|L \Phi|^2 +  \frac{1}{2}|\spartial \Phi|^2\big) dA
\]
it follows that
\begin{equation}
 \F_{t_0} [\Phi] =  \E_{S_\infty}[\Phi] -  \E_{S_{t_0}}[\Phi]  \to 0 \ .
\label{fluxdecaya}\end{equation}
We choose $t_0 > 1$ so that
\[
\F_{t_0}[\Phi]  \leq \epsilon^8 E \ .
\]
We obtain our extension of $\Phi$ to the interval $[t_0,\infty)$
from the following lemma:

\begin{lem}
Let $\Phi$ be a finite energy wave-map in $C_{[1,\infty)}$ and
$E$, $t_0$ as above. Then there
exists a wave-map extension of $\Phi$ to $\R^2 \times [t_0,\infty)$ which
has energy $E$.
\end{lem}
We remark that  as $t \to \infty$ all the energy of the extension
moves inside the cone. It is likely that an extension with this property is unique. We do not pursue this here, as it is not needed.

\begin{proof}
By pigeonholing we can choose a sequence $t_k \to \infty$
so that we have the bounds
\[
 \F_{t_k} [\Phi] \to 0 \ , \qquad t_k \int_{\partial S_{t_k}}
 |\spartial \Phi|^2 dA \to 0 \ .
\]
The second bound allows us to obtain an extension $\Phi^{(k)}[t_k]$ of the
initial data $\Phi[t_k]$ for $\Phi$ at time $t_k$ from the circle $S_{t_k}$ to all of $\R^2$ in such a way that
\[
\E[\Phi^{(k)}](t_k) -  \E_{S_{t_k}}[\Phi] \to 0 \ .
\]
By rescaling, this extension problem is equivalent to the one  in the case
of Theorem~\ref{maint}.

We solve the wave map equation backwards from time $t_k$ to time $t_0$,
with data $\Phi^{(k)}[t_k]$.  We obtain a wave map $\Phi^{(k)}$ in
the time interval $[t_0,t_k]$ which coincides with $\Phi$ in $C_{[t_0,t_k]}$.
The above relation shows that
\begin{equation}
\E[\Phi^{(k)}] \to E \ .
\label{limitenergy}\end{equation}
By energy estimates it also follows that for large $k$ and $t\in [t_0,t_k]$ we have
\begin{equation}
\E[\Phi^{(k)}](t) -  \E_{S_{t}}[\Phi^{(k)}] = \E[\Phi^{(k)}](t_k) -  \E_{S_{t_k}}[\Phi]
+ \F_{[t,t_k]}[\Phi] \lesssim \epsilon^8 E \ .
\label{noeouta}\end{equation}
Hence the energy stays small outside the cone, which by the small data
result in Theorem~\ref{tsmall} shows that no blow-up can occur outside
the cone between times $t_k$ and $t_0$.

We will obtain the extension of $\Phi$ as the
strong limit in the energy norm of a subsequence of the $\Phi^{(k)}$,
\begin{equation}
\Phi = \lim_{k \to \infty} \Phi^{(k)} \qquad \text{in}\  C(t_0,\infty; \dot{H^1})
\cap \dot C^1(t_0,\infty;L^2) \ .
\label{stronglimit}\end{equation}

We begin with the existence of a weak limit. By uniform boundedness
and the Banach-Alaoglu theorem we have weak convergence for each
fixed time
\[
\Phi^{(k)}[t] \rightharpoonup \Phi[t] \ \ \text{ weakly in } \ \dot
H^1 \times L^2 \ .
\]

Within $C_{[t_0,\infty)}$ all the $\Phi^{(k)}$'s coincide, so the
above convergence is only relevant outside the cone. But by
\eqref{noeouta}, outside the cone all $\Phi^{(k)}$ have small
energy. This places us in the context of the results in
\cite{Tataru_WM1}. Precisely, the weak stability bound
\eqref{weaklip} in Theorem~\ref{tsmall}
and an argument similar to that used in Section \ref{X_sect}
shows that the limit $\Phi$ is a regular finite energy wave-map in $[t_0,\infty)
\times \R^2$.

It remains to upgrade the convergence. On one hand,
\eqref{limitenergy} and weak convergence shows that $\E[\Phi] \leq
\E_{S_{\infty}}[\Phi]$. On the other hand $\E[\Phi] \geq
\E_{S_t}[\Phi]$, and the latter converges to
$\E_{S_{\infty}}[\Phi]$. Thus we obtain
\[
 \E[\Phi] = \E_{S_{\infty}}[\Phi] = \lim_{k\to \infty} \E[\Phi^{(k)}] \ .
\]
From weak convergence and norm convergence we obtain strong convergence
\[
 \Phi^{(k)}[t] \to \Phi[t] \ \ \text{ in } \dot H^1 \times L^2 \ .
\]
The uniform convergence in \eqref{stronglimit} follows by applying
the energy continuity in the small data result of
Theorem~\ref{tsmall} outside $C_{[t_0,\infty)}$.
\end{proof}

By energy estimates for $\Phi$ it follows that
\[
 \E[\Phi] - \E_{S_t}[\Phi] = \F_{t}[\Phi] \leq
 F_{t_0}[\Phi] \leq \epsilon^8 E \ , \qquad t \in [t_0,\infty) \ .
\]
Hence the extended $\Phi$ satisfies \eqref{casebens} so
it cannot satisfy \eqref{casebeds}. By \eqref{fluxdecay}
we obtain the sequence $(t_n,x_n)$ and $k_n \in \Z$ with
$t_n \to \infty$ so that \eqref{pconc} holds.
On the other hand from the energy-flux relation we have
\[
 \lim_{t_1,t_2 \to \infty} \F_{[t_1,t_2]}[\Phi] = 0 \ .
\]
This allows us to select again $\epsilon_n \to 0$ so that
\eqref{fluxtn} holds; clearly in this case we must also have $\epsilon_n t_n \to \infty$.
The same rescaling as in the previous subsection
leads to a sequence $\Phi^{(n)}$ of wave maps which satisfy
the conditions a)-d) above.\ret


\subsection{Elimination of the null concentration scenario}
Due to \eqref{outentnresc}, the contribution of the exterior of the cone to the
pointwise bounds for $\Phi^{(n)}$ is negligible. Precisely, for low frequencies
we have the pointwise bound
\[
 | P_k \Phi^{(n)}(1,x)| + 2^{-k} | P_k \partial_t \Phi^{(n)}(1,x)|
\lesssim   \big(2^k  (1+ 2^k |x|)^{-N}+\epsilon^4\big) E^\frac{1}{2}
, \qquad k \leq 0
\]
where the first term contains the contribution from the interior of the
cone and the second is the outside contribution. On the other hand,
for large frequencies we similarly obtain
\begin{equation}
 | P_k \Phi^{(n)}(1,x)| + 2^k | P_k \partial_t \Phi^{(n)}(1,x)|
\lesssim    \big((1+ 2^k (|x|-1)_{+})^{-N} + \epsilon^4\big)
E^\frac{1}{2}\ , \qquad k \geq 0  \ . \label{largek-point-a}
\end{equation}
Hence, in order for \eqref{pointnresc} to hold, $k_n$ must be large
enough,
\[
 2^{k_n} > m(\epsilon,E)
\]
while $x_n$ cannot be too far outside the cone,
\[
 |x_n| \leq 1 + 2^{-k_n} g(\epsilon,E) \ .
\]
This allows us to  distinguish three cases:

\begin{enumerate}[(i)]
\item {\bf Wide pockets of energy.} This is when
\[
2^{k_n} < C < \infty \ .
\]

\item {\bf Sharp time-like pockets of energy.} This is when
\[
2^{k_n}  \to \infty, \qquad |x_n| \leq \gamma < 1 \ .
\]

\item {\bf Sharp pockets of null  energy.}
 This is when
\[
2^{k_n}  \to \infty, \qquad |x_n| \to 1 \ .
\]
\end{enumerate}\ret

Our goal in this subsection is to eliminate the last case.
Precisely, we will show that:

\begin{lem}
 There exists $M = M(\epsilon,E) > 0$ and
$\gamma=\gamma(\epsilon,E) < 1$ so that for any wave map $\Phi^{(n)}$
as in \eqref{entnresc}-\eqref{pointnresc} with a small enough $\epsilon_n$ we have
\begin{equation}
2^{k_n}  > M \Longrightarrow |x_n| \leq \gamma \ .
\label{nonull}\end{equation}
\end{lem}

\begin{proof}
We apply the energy estimate \eqref{enxea}, with $\epsilon$ replaced
by $\epsilon_n$, to $\Phi^{(n)}$ in the time interval
$[\epsilon_n,1]$. This yields
\begin{equation}\label{nea}
 \int_{S_{1}}
\left( 1-|x|+\epsilon_n\right)^{-\frac12} \big(|L\Phi^{(n)}|^2 +
|\spartial \Phi^{(n)}|^2\big) dx \lesssim E \ .
\end{equation}
The relation \eqref{nonull} would follow from
the pointwise concentration bound in \eqref{pointnresc} if we can prove that for $k > 0$ the bound \eqref{nea}, together with \eqref{entnresc} and
\eqref{outentnresc} at time $t=1$, imply the pointwise estimate
\begin{equation}
 | P_k \Phi^{(n)}(1,x)| + 2^{-k} | P_k \partial_t \Phi^{(n)}(1,x)|
\lesssim  \Big( \big( (1-|x|)_+  + 2^{-k}  + \epsilon_n
\big)^\frac18 +\epsilon^2\Big) E^\frac12 \ .
\label{largek-point-b}
\end{equation}
In view of \eqref{largek-point-a} it suffices to prove this bound in
the case where $\frac12 < |x| < 2$. After a rotation we can assume
that $x=(x_1,0)$ with $\frac12 < x_1 < 2$.

In general we have
\[
 (1-|x|)_{+} \leq (1-x_1)_{+} +x_2^2 \ ,
\]
and therefore from \eqref{nea} we obtain
\begin{equation}
\int_{S_{1}} \frac{1}{ (|1-x_1|+x_2^2 + \epsilon_n)^\frac12}
\big(|L\Phi^{(n)}|^2 + |\spartial \Phi^{(n)}|^2\big) dx \lesssim E \
. \label{ned}
\end{equation}
At  the spatial location $(1,0)$ we have $\spartial \Phi^{(n)} =
\partial_2$, so  we obtain the rough bound
\[
|\partial_2 \Phi^{(n)}| \lesssim |\spartial \Phi^{(n)}|+
(|x_2|+|1-x_1|) |\nabla_x \Phi^{(n)}| \ .
\]
Similarly,
\[
|\partial_t  \Phi^{(n)}-\partial_1 \Phi^{(n)}| \lesssim |L
\Phi^{(n)}|+ (|x_2|+|1-x_1|) |\nabla_x \Phi^{(n)}| \ .
\]
Hence, taking into account \eqref{entnresc} and \eqref{outentnresc}, from \eqref{ned} we obtain
\begin{equation}
\int_{t=1} \frac{1}{ ((1-x_1)_++x_2^2 + \epsilon_n)^\frac12+
\epsilon^8} \big(|\partial_t \Phi^{(n)} -\partial_1 \Phi^{(n)}|^2 +
|\partial_2 \Phi^{(n)}|^2\big) dx \lesssim E \ . \label{neb}
\end{equation}

Next, given a dyadic frequency $2^k \geq 1$ we consider an angular
parameter $2^{-\frac{k}2} \leq \theta \leq 1$ and we rewrite the
multiplier $P_k$ in the form
\[
P_k = P_{k,\theta}^1 \partial_1 + P_{k,\theta}^2 \partial_2
\]
where $ P_{k,\theta}^1$ and $ P_{k,\theta}^2$ are multipliers with
smooth symbols supported in the sets $\{ |\xi| \approx 2^k, \
|\xi_2| \lesssim 2^k \theta\}$, respectively $\{ |\xi| \approx 2^k,
\ |\xi_2| \gtrsim 2^k \theta\}$. The size of their symbols is given
by
\[
|p_{k,\theta}^1(\xi)| \lesssim 2^{-k}, \qquad |p_{k,\theta}^2(\xi)|
\lesssim \frac{1}{|\xi_2|} \ .
\]
Therefore, they satisfy the $L^2 \to L^{\infty}$ bounds
\begin{equation}
\|  P_{k,\theta}^1\|_{L^2 \to L^\infty} \lesssim \theta^{\frac12} \
, \qquad \|  P_{k,\theta}^2\|_{L^2 \to L^\infty} \lesssim
\theta^{-\frac12}  \ . \label{lipkt}
\end{equation}
In addition, the kernels of both $ P_{k,\theta}^1$ and $
P_{k,\theta}^2$ decay rapidly on the $2^{-k} \times \theta^{-1}
2^{-k}$ scale. Thus one can add weights in \eqref{lipkt} provided
that they are slowly varying on the same scale. The weight in
\eqref{ned} is not necessarily slowly varying, but we can remedy
this by slightly increasing the denominator to obtain the  weaker
bound
\begin{equation}
\int_{t=1}\!\! \frac{1}{ ((1-x_1)_+ + x_2^2  +
2^{-k}+\epsilon_n)^\frac12+ \epsilon^8} \big(|\partial_t
\Phi^{(n)}\!\!-\partial_1 \Phi^{(n)}|^2 + |\partial_2
\Phi^{(n)}|^2\big) dx\! \lesssim\! E \ . \label{nee}
\end{equation}

Using \eqref{entnresc}, \eqref{nee} and the weighted version of \eqref{lipkt}
we obtain
\[
\begin{split}
|P_k \Phi^{(n)}(1,x)| \leq &\ | P_{k,\theta}^1 \partial_1
\Phi^{(n)}(1,x)|+ | P_{k,\theta}^2 \partial_2 \Phi^{(n)}(1,x)|
\\ \lesssim &\ \Big(\theta^\frac12  + \theta^{-\frac12}\big( ((1-x_1)_+ + x_2^2
+ 2^{-k}+\epsilon_n)^\frac14 + \epsilon^4\big) \Big) E^\frac12 \ .
\end{split}
\]
We set  $x_2= 0$ and  optimize with respect to $\theta \in [2^{-k/2},1]$
to obtain the desired bound \eqref{largek-point-b} for $\Phi^{(n)}$,
\[
|P_k \Phi^{(n)}(1,x_1,0)| \lesssim \left(\left((1-x_1)_+ + 2^{-k} +
\epsilon_n\right)^\frac18 + \epsilon^2\right) E^\frac12 \ , \qquad
\frac12 < x_1 < 2 \ .
\]
A similar argument yields
\[
2^{-k} |P_k \partial_1 \Phi^{(n)}(1,x_1,0) | \lesssim
\left(\left((1-x_1)_+ + 2^{-k} + \epsilon_n\right)^\frac18 +
\epsilon^2\right) E^\frac12 \ , \qquad \frac12 < x_1 < 2 \ .
\]
On the other hand, from \eqref{nee} we directly
obtain
\[
2^{-k} |P_k (\partial_t -\partial_1) \Phi^{(n)}(1,x_1,0)| \lesssim
\left(\left((1-x_1)_+ + 2^{-k} + \epsilon_n\right)^\frac18 +
\epsilon^2\right) E^\frac12 \ , \ \ \frac12 < x_1 < 2 \ .
\]
Combined with the previous inequality, this yields the bound in \eqref{largek-point-b} for $\partial_t \Phi^{(n)}$.
\end{proof}\ret


\subsection{Nontrivial energy in a time-like cone}

According to the previous step, the points $x_n$ and frequencies $k_n$
in \eqref{pointnresc}  satisfy  one of the following two conditions:

\begin{enumerate}[(i)]
\item {\bf Wide pockets of energy.} This is when
\[
c(\epsilon,E) < 2^{k_n}  < C(\epsilon,E) \ .
\]

\item {\bf Sharp time-like pockets of energy.} This is when
\[
2^{k_n}  > C(\epsilon,E) \ , \qquad |x_n| \leq \gamma(\epsilon,E) <
1 \ .
\]
\end{enumerate}\ret

Using only the bounds \eqref{entnresc} and \eqref{outentnresc}, we
will prove that there exists $\gamma_1=\gamma_1(\epsilon,E) < 1$ and
$E_1=E_1(\epsilon,E) > 0$ so that in both cases there is some amount
of uniform time-like energy concentration,
\begin{equation}
\frac12\int_{t=1, |x| < \gamma_1}  \big( |\partial_t \Phi^{(n)}|^2 +
|\nabla_x \Phi^{(n)}|^2\big) dx \geq E_1 \ . \label{energyinsidea}
\end{equation}
For convenience we drop the index $n$ in the following computations.
Denote
\[
E(\gamma_1) = \frac12\int_{t=1, |x| < \gamma_1}  \big( |\partial_t
\Phi|^2 + |\nabla_x \Phi|^2\big) dx
\]
where $\gamma_1\in (\frac{\gamma+1}2 ,1)$ will  be chosen later.
Here $\gamma$ is from line \eqref{nonull} of the previous
subsection. Thus at time $1$ the function $\Phi$ has energy
$E(\gamma_1)$ in $\{ |x| < \gamma_1\}$, energy $\leq E$ in
$\{\gamma_1 < |x| <1\}$, and energy $\leq \epsilon^8 E$ outside the
unit disc. Then we obtain different pointwise estimates for $P_k
\Phi[1]$ in two main regimes:\ret

{\bf (a) $2^k (1-\gamma_1) < 1$.}
Then we obtain
\[
\|P_k \Phi(1)\|_{L^\infty_x} + 2^{-k} \|P_k \partial_t
\Phi(1)\|_{L^\infty_x} \lesssim  E(\gamma_1)^\frac12 + \left(
(2^k(1-\gamma_1))^\frac12  +\epsilon^4\right)  E^\frac12 \ ,
\]
with a further improvement if both

{\bf (a1)} $2^k (1-\gamma_1) < 1 < 2^k (1-\gamma)$ and $|x| < \gamma$,
namely
\[
|P_k \Phi(1,x)| + 2^{-k} |P_k \partial_t \Phi(1,x)| \lesssim
E(\gamma_1)^\frac12 + \left((2^k(1-\gamma_1))^\frac12
(2^k(1-\gamma))^{-N} +\epsilon^4\right)  E^\frac12 \ .
\]\ret

{\bf (b) $2^k (1-\gamma_1) \geq 1$.}
Then
\[
|P_k \Phi(1,x)| + 2^{-k} |P_k \partial_t \Phi(1,x)| \lesssim
E(\gamma_1)^\frac12 + \left( (2^k (1-\gamma))^{-N}
+\epsilon^4\right) E^\frac12, \qquad |x| < \gamma \ .
\]\ret

We use these estimates to bound $E(\gamma_1)$ from below. We first
observe that if $1-\gamma_1$ is small enough then case (i) above
implies we are in regime {\bf (a)}, and from \eqref{pointnresc} we
obtain
\[
\epsilon \lesssim  E(\gamma_1)^\frac12 + \Big(\big(C(\epsilon,E)
(1-\gamma_1)\big)^\frac12  +\epsilon^4\Big)  E^\frac12 \ ,
\]
which gives a bound from below for $E(\gamma_1)$ if $1-\gamma_1$
and $\epsilon$ are small enough.

Consider now the remaining case (ii) above. If we are in regime {\bf
(a)} but not {\bf (a1)}, then the bound in {\bf (a)} combined with
\eqref{pointnresc} gives
\[
\epsilon \lesssim E(\gamma_1)^\frac12 +
\Big(\frac{(1-\gamma_1)^\frac12}{(1-\gamma)^\frac12} + \epsilon^4
\Big)E^\frac12 \ ,
\]
which suffices if $1-\gamma_1$ is small enough. If we are in regime
{\bf (a1)} then we obtain exactly the same inequality directly.
Finally, if we are in regime {\bf (b)} then we achieve an even
better bound
\[
\epsilon \lesssim E(\gamma_1)^\frac12 +
\Big(\Big(\frac{1-\gamma_1}{1-\gamma} \Big)^N  + \epsilon^4
\Big)E^\frac12 \ .
\]
Thus, \eqref{energyinsidea} is proved in all cases for a small
enough $1-\gamma_1$.\ret


\subsection{Propagation of time-like energy concentration}

Here we use the flux relation \eqref{fluxtn} to propagate the
time-like energy concentration in \eqref{energyinsidea}
\emph{uniformly} to smaller times $t \in
[\epsilon_n^\frac12,\epsilon_n^\frac14]$. Precisely, we show that
there exists $\gamma_2 = \gamma_2(\epsilon,E) < 1$ and $E_2 =
E_2(\epsilon,E) > 0$ so that
\begin{equation}
  \frac12\int_{ |x| < \gamma_2 t}
   \big( |\partial_t\Phi^{(n)} |^2 + |\nabla_x \Phi^{(n)}|^2\big)dx \geq E_2 \ ,
\qquad t \in [\epsilon_n^\frac12,\epsilon_n^\frac14] \ .
\label{enprop}
\end{equation}
At the same time, we also obtain uniform weighted $L^2_{t,x}$ bounds
for $X_0 \Phi^{(n)}$ outside smaller and smaller neighborhoods of
the cone, namely
\begin{equation}
\int_{C_{[ \epsilon_n^\frac12,\epsilon_n^\frac14 ]}^{\epsilon_n}}
\rho^{-1} |X_0 \Phi^{(n)}|^2 dx dt \lesssim E  \ . \label{xophina}
\end{equation}

The latter bound \eqref{xophina} is a direct consequence of
\eqref{xphi}, so we turn our attention to \eqref{enprop}. Given a
parameter $\gamma_2=\gamma_2(\gamma_1,E_1)$, and any $t_0 \in
[\epsilon^\frac12_n,\epsilon^{\frac14}_n]$,  we define $\delta_0$
and $\delta_1$ according to
\[
(1-\gamma_2)t_0 =  \delta_0 \ll \delta_1 \leq t_0 \ .
\]
We apply Proposition~\ref{penergyup} to $\Phi^{(n)}$ with this set
of small constants.

 Optimizing the right
hand side in \eqref{energyup} with respect to the choice of $\delta_1$ it follows
that
\[
\int_{S_1^{t_0}}  {}^{(X_0)}\!\!P_0[\Phi^{(n)}] \ dx \lesssim
\int_{S_{t_0}^{\delta_0}}  {}^{(X_0)}\!\!P_0[\Phi^{(n)}] \ dx +
|\ln(t_0/\delta_0)|^{-1} E \ .
\]
Converting the $X_0$ momentum density into the $\partial_t$ momentum
density it follows that
\[
(1-\gamma_1)^\frac12 \int_{S_1^{1-\gamma_1}}
{}^{(\partial_t)}\!\!P_0[\Phi^{(n)}] \ dx \lesssim
(1-\gamma_2)^{-\frac12}\int_{S_{t_0}^{\delta_0}}
{}^{(\partial_t)}\!\!P_0[\Phi^{(n)}] \ dx + |\ln(1-\gamma_2)|^{-1} E
\ .
\]
Hence by \eqref{energyinsidea} we obtain
\[
 (1-\gamma_1)^\frac12 E_1 \lesssim (1-\gamma_2)^{-\frac12}
 \E_{S_{t_0}^{\delta_0}}[\Phi^{(n)}]+ |\ln(1-\gamma_2)|^{-1} E \ .
\]
We choose $\gamma_2$ so that
\[
|\ln(1-\gamma_2)|^{-1} E \ll (1-\gamma_1)^\frac12 E_1 \ .
\]
Then the second right hand side term in the previous inequality
can be neglected, and for
\[
0 < E_2 \ll (1-\gamma_1)^\frac12 (1-\gamma_2)^\frac12 E_1
\]
we obtain \eqref{enprop}.\ret


\subsection{Final rescaling} The one bound concerning the rescaled wave maps $\Phi^{(n)}$
which is not yet satisfactory is \eqref{enprop}, where we would like
to have decay in $n$ instead of uniform boundedness. This can be
achieved by further subdividing the time interval
$[\epsilon_n^\frac12,\epsilon_n^\frac14]$.

For $ 2 < N < \epsilon^{-\frac14}$ we divide the time interval $[
\epsilon_n^\frac12,\epsilon_n^\frac14]$ into about $
|\ln\epsilon_n|/\ln N$ subintervals of the form $[t,Nt]$. By
pigeonholing, there exists one such subinterval which we denote by
$[t_n,Nt_n]$ so that
\begin{equation}
\dint_{C_{[t_n,Nt_n]}^{ \epsilon_n}} \frac{1}{\rho} |X_0
\Phi^{(n)}|^2 dx dt \lesssim  \frac{\ln N}{|\ln \epsilon_n|} E \ .
\label{xphidecay}
\end{equation}
We assign to $N=N_n$ the value
\[
N_n = e^{\sqrt |\ln \epsilon_n|} \ .
\]
Rescaling the wave maps $\Phi^{(n)}$ from the time interval
$[t_n,N_n t_n]$ to the time interval $[1,N_n]$ we obtain a final
sequence of rescaled wave maps, still  denoted  by $\Phi^{(n)}$,
defined on increasing sets $C_{[1,T_n]}$, where $T_n\to\infty$, with
the following properties:

\begin{enumerate}[a)]
\item Bounded energy,
\begin{equation}
 \E_{S_t}[\Phi^{(n)}](t) \approx E \ , \qquad t \in [1,T_n] \ .
\label{finala}
\end{equation}

\item Uniform amount of nontrivial time-like energy,
\begin{equation}
\E_{S_t^{(1-\gamma_2)t}}[\Phi^{(n)}](t) \geq E_2 \ , \qquad t \in
[1,T_n] \ .  \label{finalb}
\end{equation}

\item Decay to self-similar mode,
\begin{equation}
\dint_{C_{[1,T_n]}^{ \epsilon_n^\frac12}} \frac{1}{\rho} |X_0
\Phi^{(n)}|^2 dx dt \lesssim  |\log \epsilon_n|^{-\frac12} E
 \ . \label{finalc}
\end{equation}
\end{enumerate}\ret


\subsection{Concentration scales}

We partition the set $C^1_{[1,\infty)}$ into dyadic subsets
\[
C_j = \{(t,x) \in C^1_{[1,\infty)}; \ 2^{i} < t < 2^{i+1}\} \ ,  \qquad j \in \N \ .
\]
We also consider slightly larger sets
\[
 \tC_j =  \{(t,x) \in C^\frac{1}{2}_{[\frac{1}{2},\infty)}; \ 2^{i} < t < 2^{i+1}\}\ ,  \qquad j \in \N \ .
\]
Then we prove that

\begin{lem} \label{lcscales}
Let $\Phi^{(n)}$ be a sequence of wave maps satisfying \eqref{finala}, \eqref{finalb}
and \eqref{finalc}. Then for each $j \in \N$  one of the
following alternatives must hold on a subsequence:

\begin{enumerate}[(i)]
\item(Concentration of non-trivial energy) There exist  points $(t_n,x_n) \in \tC_j$, a sequence of
scales $r_n \to 0$,  and some $r=r_j$ with $ 0 < r < \frac{1}{4}$ so that the following
three bounds hold:
\begin{align}
    \E_{B(x_n,r_n)}[\Phi^{(n)}](t_n) &= \frac{1}{10} E_0 \ ,
    \label{65} \\
    \E_{B(x,r_n)}[\Phi^{(n)}](t_n) &\leq \frac{1}{10} E_0 \ , \qquad x \in B(x_n,r) \ ,
    \label{66} \\
    r_n^{-1} \int_{t_n-r_n/2}^{t_n+r_n/2} \int_{B(x_n,r)} |X_0 \Phi^{(n)}|^2 dx dt &\to 0 \ .
    \label{67}
\end{align}

\item (Nonconcentration of uniform energy) There exists some $r=r_j$ with $ 0 < r < \frac{1}{4}$ so that
the following three bounds hold:
\begin{align}
    \E_{B(x,r)}[\Phi^{(n)}](t) &\leq \frac{1}{10} E_0 \ , \qquad \forall (t,x) \in C_j \ ,
    \label{68}\\
    \E_{S_t^{(1-\gamma_2)t}}[\Phi^{(n)}](t) &\geq E_2 \ , \qquad \hbox{when\ \ }
    B\big(0,(1-\gamma_2)t\big)\subseteq C^1_{[1,\infty)} \ , \label{69}\\
    \dint_{C_j} |X_0 \Phi^{(n)}|^2 dx dt &\to 0 \ .
    \label{70}
\end{align}
\end{enumerate}
\end{lem}\ret

\begin{proof}
The argument boils down to some straightforward pigeonholing, and is essentially
identical for all $j\in\N$, which is now fixed throughout the proof.
Given any large parameter $N \in \N$
we partition the time interval $[2^j,2^{j+1}]$ into about $N2^j $ equal intervals,
\[
 I_k = [2^j+(k-1)/(10N),2^j+k/(10N)], \qquad k = 1,10N2^j \ .
\]
Then it suffices to show that the conclusion of the lemma holds
with $C_j$ and $\tC_j$  replaced by
\[
 C_j^k = C_j \cap I_k \times \R^2, \qquad  \tC_j^k = \tC_j \cap I_k \times \R^2 \ .
\]
We begin by constructing a low energy barrier around $C_j^k$.
To do this we partition $\tC_j^k\setminus C_j^k$ into
$N$ sets
\[
    \tC_j^{k,l} = \left\{ (t,x) \in \tC_j^k; \  \frac{1}{4}+\frac{l-1}{4N}
    < t-|x| < \frac{1}{4}+\frac{l}{4N}\right\} \ , \quad j=1,N \ .
\]
By integrating energy estimates we have
\[
    \sum_{l=1}^{N} \int_{I_k} \E_{\tC_j^{k,l} }[\Phi^{(n)}](t) dt
    \leq \int_{I_k}\E_{ \tC_j^k}[\Phi^{(n)}](t) dt \leq \frac{1}{10N} E \ .
\]
Thus by pigeonholing, for each fixed $n$ there must exist $l_n$ so that
\[
 \sum_{l=l_n-1}^{l_n+1} \int_{I_k} \E_{ \tC_j^{k,l}}[\Phi^{(n)}](t) dt
 \leq \frac{3}{10N^2} E \ ,
\]
and further there must be some $t_n \in I_k$ so that
\[
 \sum_{j=j_n-1}^{j_n+1} \E_{\tC_j^{k,l} }[\Phi^{(n)}](t_n)
 \leq \frac{3}{N} E \ .
\]
For $t \in I_k$ we have $|t-t_n| < 1/(10N)$, and therefore the $t$ section of
$\tC_j^{k,l_n}$ lies within the influence cone of the $t_n$ section
of $\tC_j^{k,l_n-1} \cap  \tC_j^{k,l_n}  \cap \tC_j^{k,l_n+1} $.
Hence it follows that one has the uniform bound
\[
 \E_{ \tC_j^{k,l_n}}[\Phi^{(n)}(t)] \leq \frac{3}{N} E, \qquad t \in I_k \ .
\]
We choose $N$ large enough so that we beat the perturbation energy
\[
 \frac{3}{N} E \leq \frac{1}{20} E_0 \ .
\]
Then the set $\tC_j^{k,l_n}$ acts as an energy barrier for $\Phi^{(n)}$ within
$\tC_j^{k}$, separating the evolution inside from the evolution outside
with a small data region. We denote the inner region by
$\tC_j^{k,<l_n}$ and its union with $\tC_j^{k,l_n}$ by
$\tC_j^{k,\leq l_n}$.  We fix $r_0$ independent of $n$ so that
\[
 (t,x) \in \tC_j^{k,<l_n} \Longrightarrow \{t\} \times B(x,4r_0)
\subset  \tC_j^{k,\leq l_n} \ .
\]

To measure the energy concentration in balls we define the functions
\[
f_n: [0,r_0] \times I_k \to \R^+ \ , \qquad
f_n(r,t) = \sup_{\{x;\{t\} \times B(x,r) \subset  \tC_j^{k,\leq l_n}\}}
\E_{B(x,r)}[\Phi^{(n)} ](t) \ .
\]
The functions $f_n$ are continuous in both variables and nonincreasing with respect to $r$.  We also define the functions
\[
    r_n: I_k \to (0,r_0] \ , \qquad r_n(t) =
    \begin{cases}
        \inf \{r \in [0,r_0];\ f_n(t,r) \geq \frac{E_0}{10}\}\ ,
        & \text{if } f_n(t,r_0) \geq \frac{E_0}{10};\\
        r_0  \ , & \text{otherwise}.
    \end{cases}
\]
which measure the lowest spatial scale on which concentration occurs
at time $t$. Due to the finite speed of propagation it follows that
the $r_n$ are Lipschitz continuous with Lipschitz constant $1$,
\[
|r_n(t_1) - r_n(t_2)| \leq |t_1-t_2| \ .
\]

The nonconcentration estimate \eqref{68} in case ii) of the Lemma corresponds to the case
when all functions $r_n$ admit a common strictly positive lower bound (note that \eqref{finalb} and \eqref{finalc}
give the other conclusions).

It remains to consider the case when on a subsequence we have
\[
 \lim_{n \to \infty} \inf_{I_k} r_n = 0
\]
and show that this yields the concentration scenario i).
We denote ``kinetic energy'' in $\tC_j$ by
\[
\alpha_n^2 = \int_{\tC_j}  |X_0 \Phi^{(n)}|^2 dxdt \ .
\]
By \eqref{finalc} we know that $\alpha_n \to 0$.
Using $\alpha_n$ as a threshold for the concentration functions $r_n$, after passing
to a subsequence we must be in one of the following three cases:\ret

\case{1}{$r_n$ Dominates} For each $n$ we have $ r_n(t) >\alpha_n $
in $I_k$. Then we let $t_n$ be the minimum point for $r_n$ in $I_k$
and set $r_n^0 = r_n(t_n) \to 0$. By definition we have
$f_n(t_n,r_n(t_n)) = E_0/10$. We choose a point $x_n$ where the
maximum of $f_n(t_n,r_n(t_n))$ is attained. This directly gives
\eqref{65}. For \eqref{66} we observe that, due to the existence of
the energy barrier, $x_n$ must be at least at distance  $3r_0$ from
the lateral boundary of $\tC_j^{k,\leq l_n}$. Hence, if $x \in
B(x_n,r_0)$ then $x$ is at distance at least $2 r_0$ from  $\partial
\tC_j^{k,\leq l_n}$ and \eqref{66} follows. For \eqref{67} it
suffices to know that $r_n(t_n)^{-1} \alpha_n^2 \to 0$, which is
straightforward from the assumptions of this case.\ret

\case{2}{Equality}  For each $n$ there exists $t_n \in I_k$ such that $ \alpha_n = r_n(t_n)$.
This argument is a repeat of the previous to the one, given that we define $r_n^0 = r_n(t_n)$ and set
up the estimate \eqref{66} around this $t_n$ as opposed to the minimum of $r_n(t)$.\ret

\case{3}{$\alpha_n$ Dominates} For each $n$ we have $r_n(t) < \alpha_n$ in $I_k$.
For $t \in I_k$ set
\[
 g(t) = \int_{\tC_j^{k,\leq l_n}}  |X_0 \Phi^{(n)}(t)|^2 dx \ .
\]
Then by definition
\[
 \int_{I_k} g(t) = \alpha_n^2 \ .
\]
Let $\td{I}_k$ be the middle third of $I_k$, and consider the localized averages
\[
\I = \int_{\td{I}_k} \frac{1}{r_n(t)} \int_{t-r_n(t)/2}^{t+r_n(t)/2} g(s) ds dt \ .
\]
Since $r_n(t)$ is Lipschitz with Lipschitz constant $1$, if
$s \in [t-r_n(t)/2,t+r_n(t)/2]$ then $\frac12 r_n(s) <   r_n(t) <  2r_n(s)$
and  $t \in [s-r_n(s),s+r_n(s)]$. Hence changing the order of integration
in $\I$ we obtain
\[
 \I \leq 2 \int_{\td{I}_k}
\int_{t-r_n(t)/2}^{t+r_n(t)/2} \frac{1}{r_n(s)} g(s) ds dt \leq 4 \int_{I_k} g(s) ds
= 4 \alpha_n^2 \ .
\]
Hence by pigeonholing there exists some $t_n \in \td{I}_k$ so that
\[
  \frac{1}{r_n(t_n)} \int_{t_n-r_n(t_n)/2}^{t_n+r_n(t_n)/2} g(s) ds \leq
4 \alpha_n^2 |\td{I}_k|^{-1}
\]
Then let $x_n$ be a point where the supremum in the definition of
$f(t_n,r_n)$ is attained. The relations \eqref{65}-\eqref{68}
follow as above.
\end{proof}\ret


\subsection{The compactness argument}

To conclude the proof of the Theorems~\ref{maint},\ref{maints} we consider
separately the two cases in Lemma~\ref{lcscales}:\ret

(i) \textbf{Concentration on small scales.} Suppose that the alternative (i)
in  Lemma \ref{lcscales} holds for some $j \in \N $.
On a subsequence we can assume that $(t_n,x_n) \to (t_0,x_0) \in \tC_j$.
Then we define the rescaled wave maps
\[
 \Psi^{(n)}(t,x) = \Phi^{(n)}(t_n+r_n t, x_n+r_n x)
\]
in the increasing sets $B(0,r_0/r_n) \times [-\frac12,\frac12]$
They have the following properties:

\begin{enumerate}[a)]
\item Bounded energy,
\[
\E[\Psi^{(n)}](t) \leq \E[\Phi] \ , \qquad t\in [-\frac12,\frac12] \ .
\]

\item Small energy in each unit ball,
\[
\sup_x \E_{B(x,1)}[\Psi^{(n)}](t) \leq \frac{1}{10} E_0 \ , \qquad t\in [-\frac12,\frac12] \ .
\]

\item Energy concentration in the unit ball centered at $t=0,x=0$,
\[
 \E_{B(0,1)}[\Psi^{(n)}](0) = \frac{1}{10} E_0 \ .
\]

\item Time-like energy decay: There exists a constant time-like vector $X_0(t_0,x_0)$
such that for each $x$ we have
\[
\dint_{[-\frac12,\frac12]\times B(x,1)} |X_0(t_0,x_0) \Psi^{(n)} |^2 dx dt \to 0 \ .
\]
\end{enumerate}
Note  in part (d) we used the fact that $(t_n,x_n) \to (t_0,x_0)$.

By the
compactness result in Proposition~\ref{pcompact} it follows that on a
subsequence we have strong uniform convergence on compact sets,
\[
\Psi^{(n)} \to \Psi \qquad \text{ in } H^{1}_{loc}\Big( B(0,r_0/(2r_n) \times [-\frac12,\frac12]\Big)
\]
where $\Psi \in H^{\frac32-\epsilon}_{loc}$ is a wave-map. Thus, we have obtained a wave map
$\Psi$ defined on all of $[-\frac12,\frac12] \times \R^2$, with the additional properties that
\[
\frac{1}{10} E_0 \leq  \E[\Psi] \leq E[\Phi]
\]
and
\[
X_0(t_0,x_0) \Psi = 0 \ .
\]
Then $\Psi$ extends uniquely to a wave map in $\R \times \R^2$ with
the above properties (e.g by transporting its values along the flow of $X_0(t_0,x_0)$).
After a Lorentz transform that takes $
X_0(t_0,x_0)$ to $\partial_t$ the function $\Psi$ is turned into a
nontrivial finite energy harmonic map with energy bound
$\E[\Psi] \leq E[\Phi]$.\ret

(ii) \textbf{Nonconcentration.} Assume now that the alternative (ii)
in  Lemma~\ref{lcscales} holds for \emph{every} $j \in \N$.
 There there is no need to rescale.
Instead, we successively use directly the compactness result in Proposition~\ref{pcompact}
in the interior of each set $C_j\cap C_{[2,\infty)}^2$. We obtain strong convergence on a subsequence
\[
\Phi^{(n)} \to \Psi \qquad \text{ in } H^{1}_{loc}(C_{[2,\infty)}^2)
\]
with $\Psi \in H^{\frac32-\epsilon}_{loc}(C_{[2,\infty)}^2)$.
From \eqref{69} and energy bounds (e.g. \eqref{finalb}) we obtain
\[
0 < E_2 \leq \sup_{t\geqslant 2}\E_{B(0,t-2)}[\Psi](t) \leq \E[\Phi] \ .
\]
From \eqref{70}  it follows also that
\[
X_0 \Psi = 0 \ .
\]
By rescaling (i.e. extending $\Psi$ via homogeneity), we may replace 
the interior of the translated cone $C_{[2,\infty)}^2$ with the interior of the full cone
$t>r$ and retain the assumptions on $\Psi$, in particular that it is non-trivial
with finite energy up to the boundary $t=r$.
But this contradicts Theorem \ref{ss_thm}, and therefore shows that
scenario (i) above is in fact the only alternative.


\bibliography{wm}
\bibliographystyle{plain}
\end{document}